\documentclass[review]{elsarticle}
\usepackage{lineno,hyperref}
\usepackage{amsfonts}
\usepackage{amsmath}
\usepackage{amsthm}
\usepackage{amssymb}
\usepackage{amscd,multirow}
\usepackage{graphicx}
\usepackage{epstopdf}
\usepackage[titletoc]{appendix}
\usepackage{subeqnarray,cases}
\usepackage[ruled]{algorithm2e}
\newtheorem{theorem}{Theorem}[section]
\newtheorem{corollary}{Corollary}[section]
\newtheorem{lemma}{Lemma}[section]
\newtheorem{proposition}{Proposition}[section]
\newtheorem{remark}{Remark}[section]

\allowdisplaybreaks[4]

\modulolinenumbers[5]

\begin{document}

\begin{frontmatter}

\title{ Perturbed inertial primal-dual dynamics with damping and scaling terms  for linearly constrained convex optimization problems \tnoteref{mytitlenote}}

\tnotetext[mytitlenote]{This work was supported by  the National Science  Foundation of China
(11471230) and the Scientific Research Foundation  of the Education Department of Sichuan Province (16ZA0213).}

\author[mymainaddress]{Xin He}
\ead{hexinuser@163.com}

\author[mysecondaryaddress]{Rong Hu}
\ead{ronghumath@aliyun.com}

\author[mymainaddress]{Ya-Ping Fang\corref{mycorrespondingauthor}}
\cortext[mycorrespondingauthor]{Corresponding author}
\ead{ypfang@aliyun.com}

\address[mymainaddress]{Department of Mathematics, Sichuan University, Chengdu, Sichuan, P.R. China}
\address[mysecondaryaddress]{Department of Applied Mathematics, Chengdu University of Information Technology, Chengdu, Sichuan, P.R. China}

\begin{abstract}
We propose a perturbed inertial primal-dual dynamic with damping and scaling coefficients, which involves inertial terms both for primal and dual variables, for a linearly constrained convex optimization problem in a Hilbert setting. With  different choices of damping and  scaling coefficients, by a Lyapunov analysis approach we discuss the  asymptotic properties of the dynamic and prove its fast convergence properties.
Our results  can be viewed extensions of the existing ones on inertial dynamical  systems for the unconstrained convex optimization problem to the linearly  constrained convex optimization problem.
\end{abstract}

\end{frontmatter}


\section{Introduction}
\subsection{Problem statement}
 Let $\mathcal{H}_1$ and $\mathcal{H}_2$ be two real Hilbert spaces with inner $\langle \cdot,\cdot\rangle$ and norm $\|\cdot\|$. Let $f: \mathcal{H}_1\to\mathbb{R}$ be a  differentiable  convex  function and $A:\mathcal{H}_1\to\mathcal{H}_2$ be a continuous linear operator with its adjoint operator $A^T$.  Consider the  perturbed inertial primal-dual dynamical system
\begin{equation}\label{dy_dynamic}
	\begin{cases}
		\ddot{x}(t)+\alpha(t)\dot{x}(t) = -\beta(t)(\nabla f(x(t))+A^T(\lambda(t)+\delta(t)\dot{\lambda}(t))+\sigma A^T(Ax(t)-b))+\epsilon(t) ,\\
		\ddot{\lambda}(t)+\alpha(t)\dot{\lambda}(t) = \beta(t)(A(x(t)+\delta(t)\dot{x}(t))-b)
	\end{cases}
\end{equation}
where $t\in [t_0,+\infty)$ with $t_0\geq 0$, $\sigma\geq 0$, $\alpha:[t_0,+\infty)\to(0,+\infty)$  is a viscous damping coefficient, $\beta:[t_0,+\infty)\to (0,+\infty)$ is a scaling coefficient,   $\delta: [t_0,+\infty)\to (0,+\infty)$ is an extrapolation coefficient, and $\epsilon:[t_0,+\infty)\to\mathcal{H}_1$ is an integrable source term that can be interpreted as a small external perturbation.
In terms of the dynamic \eqref{dy_dynamic},  in this paper, we shall develop a fast primal-dual dynamic approach to solve the  linearly constrained convex optimization problem
	\begin{equation}\label{question}
				\min_{x}  \quad f(x), \quad s.t.  \  Ax = b.
	\end{equation}
The primal-dual dynamic \eqref{dy_dynamic} involves three important parameters: the damping coefficient $\alpha(t)$, the extrapolation coefficient $\delta(t)$, and the scaling coefficient $\beta(t)$, which play crucial roles in deriving the fast convergence properties. The importance of the damping coefficient and the scaling coefficient has been widely recognized in  inertial dynamical approaches  \cite{AttouchC2017,CabotEG2009,Su2014}  as well as fast  algorithms  \cite{Nesterov2013,Su2014, AttouchCRF2019,Shi2018,Wibisono2016} for  unstrained optimization problems.  
Recently, the damping technique and  the scaling technique were also used to develop  inertial primal-dual dynamic approaches  and inertial primal-dual algorithms for linearly constrained optimization problems, see \cite{Zeng2019,HeHF2020,HeHF2021F,HeHF2021C}. Extrapolation coefficients were also considered  in \cite{Zeng2019,HeHF2020}.

Let $\mathcal{L}(x,\lambda)$ and  $\mathcal{L}^{\sigma}(x,\lambda)$ be the Lagrangian function and the augmented Lagrangian function of the problem  \eqref{question}   respectively, i.e.,
\[\mathcal{L}(x,\lambda) = f(x)+\langle \lambda,Ax-b\rangle\]
and
\begin{equation}\label{AugL}
	\mathcal{L}^{\sigma}(x,\lambda) =\mathcal{L}(x,\lambda) +\frac{\sigma}{2}\|Ax-b\|^2= f(x)+\langle \lambda,Ax-b\rangle +\frac{\sigma}{2}\|Ax-b\|^2,
\end{equation}
where $\sigma\geq 0$ is the penalty parameter and $\lambda$ is the Lagrangian multiplier.
Let  $\Omega\subset \mathcal{H}_1\times\mathcal{H}_2$  be the  saddle point set of $\mathcal{L}$ ($\mathcal{L}^{\sigma}$).  It is known that  $(x^*,\lambda^*)\in\Omega$ if and only if
\begin{equation}\label{saddle_point}
	\begin{cases}
		-A^T\lambda^* =  \nabla f(x^*),\\
		Ax^* -b =0.
	\end{cases}
\end{equation}

Throughout this paper, we always assume that $f$  is a convex continuously differentiable function and  $\Omega\neq\emptyset$.  We will investigate the asymptotical behavior of  the dynamic \eqref{dy_dynamic}  with the damping coefficient $\alpha(t)=\frac{\alpha}{t^r}$ and  the extrapolation coefficient $\delta(t)=\delta t^s$, where $\alpha>0$, $\delta>0$, and $0\leq r\leq s\leq 1$.

\subsection{Related works}

\subsubsection{Inertial dynamical systems with  damping coefficients}

Let's recall some important inertial dynamical systems with  damping coefficients for the unstrained optimization problem 
\begin{equation}\label{min_fun}
   \min \Phi(x),
\end{equation}
where $\Phi(x)$ is a smooth convex function.   The following inertial gradient system:
\begin{equation*}
 (\text{IGS}_{\alpha})\qquad	\ddot{x}(t)+\alpha(t)\dot{x}(t)+\nabla \Phi(x(t))=0,
\end{equation*}
and its perturbed version
\begin{equation*}
 (\text{IGS}_{\alpha,\epsilon})\qquad	\ddot{x}(t)+\alpha(t)\dot{x}(t)+\nabla \Phi(x(t))=\epsilon(t),
\end{equation*}
have been intensively studied in the literature.   When damping coefficient $\alpha(t)=\alpha$ with $\alpha>0$: $(\text{IGS}_{\alpha})$ becomes the heavy ball with friction system, which was introduced by Polyak \cite{Polyak1964}, and the asymptotic behavior has been investigated in \cite{Alvarez2000,Begout2015}; under the assumption $\int^{+\infty}_{t_0}\|\epsilon(t)\|dt<+\infty$, Haraux and Jendoubi \cite{HarauxJ2012} studied the asymptotic behavior of solutions of $(\text{IGS}_{\alpha,\epsilon})$. When $\alpha(t)=\frac{\alpha}{t^r}$ with $\alpha>0,\ r\in(0,1)$:  Cabot and Frankel \cite{CabotF2012} and May \cite{May2015} investigated the asymptotic behavior of $(\text{IGS}_{\alpha})$ as $t$ goes to infinity; Jendoubi and May \cite{JendoubiM2015} generalized the results of \cite{CabotF2012} to $(\text{IGS}_{\alpha,\epsilon})$ with $\int^{+\infty}_{t_0}\|\epsilon(t)\|dt<+\infty$ and $\int^{+\infty}_{t_0}t\|\epsilon(t)\|dt<+\infty$ respectively; Balti and May \cite{Balti2016} obtained the $\mathcal{O}(1/t^{2r})$ convergence rate with $\int^{+\infty}_{t_0}t^r\|\epsilon(t)\|dt<+\infty$ and the $o(1/t^{1+r})$ convergence rate with $\int^{+\infty}_{t_0}t^{(1+r)/2}\|\epsilon(t)\|dt<+\infty$ for $(\text{IGS}_{\alpha,\epsilon})$; Sebbouh et al. \cite{Sebbouh2020} investigated the convergence rate of the values along the trajectory of  $(\text{IGS}_{\alpha,\epsilon})$  under some additional geometrical conditions on $\Phi(x)$.  When $\alpha(t)=\frac{\alpha}{t}$: Su et al. \cite{Su2014} pointed out that $(\text{IGS}_{\alpha})$ with $\alpha=3$ can be viewed as a continuous version of  the Nesterov's accelerated gradient algorithm (\cite{Beck2009,Nesterov1983}), and obtained the convergence rate $\Phi(x(t))-\min \Phi=\mathcal{O}(1/t^2)$ as $\alpha\geq 3$; Attouch et al. \cite{AttouchCPR2018} investigated the asymptotic behavior of $(\text{IGS}_{\alpha,\epsilon})$ as $\alpha\geq 3$ under the assumption $\int^{+\infty}_{t_0}t\|\epsilon(t)\|dt<+\infty$;   May \cite{May2017} proved an improved convergence rate  $\Phi(x(t))-\min \Phi= o(1/t^2)$ with  $\alpha > 3$; in the case $\alpha\leq3$ of $(\text{IGS}_{\alpha})$ and $(\text{IGS}_{\alpha,\epsilon})$, the $\mathcal{O}(1/t^{2\alpha/3})$ rate of convergence  can be found in \cite{AttouchCRR2019,Vassilis2018}; the optimal convergence rates under some additional geometrical conditions was studied by \cite{Aujol2019} for $(\text{IGS}_{\alpha})$ with $\alpha>0$.  For general damping coefficient $\alpha(t)$, it has been  investigated by \cite{AttouchC2017,AttouchCCRR2018,CabotEG2009}.

\subsubsection{Inertial dynamical systems with scaling coefficients}
Balhag el al. \cite{Balhag2020} considered following inertial gradient system with time scaling and constant  damping coefficient:
\begin{equation}\label{dy_heavy}
	\ddot{x}(t)+\alpha\dot{x}(t)+\beta(t)\nabla \Phi(x(t))=0,
\end{equation}
for solving problem \eqref{min_fun}, under the assumption $\beta(t)=e^{\beta t}$ with $\beta\leq \alpha$, they can obtain the linear convergence without strong convexity of $\Phi$. From the calculus of variations, Wibisono et al. \cite{Wibisono2016} proposed the following dynamic
\begin{equation}\label{dy_Wib}
	\ddot{x}(t)+\frac{\alpha}{t}\dot{x}(t)+C(\alpha-1)^2t^{\alpha-3}\nabla \Phi(x(t))=0,
\end{equation}
with time scaling $\beta(t)=C(\alpha-1)^2t^{\alpha-3}$ for problem \eqref{min_fun}  where $\alpha> 1$ and $C>0$, and obtained the $\mathcal{O}(1/t^{\alpha-1})$ rate of convergence.  Fazlyab et al. \cite{Fazlyab2017} extended the dynamic \eqref{dy_Wib} to following dual dynamic for solving problem \eqref{question} :
\begin{equation*}
	\ddot{\lambda}(t)+\frac{\alpha}{t}\dot{\lambda}(t)+C(\alpha-1)^2t^{\alpha-3}\nabla \mathbb{G}(\lambda(t))=0,
\end{equation*}
where $\mathbb{G}(\lambda)= \min_{x} \mathcal{L}(x,\lambda)$, $\alpha> 1$ $C>0$,  the convergence rate $\mathbb{G}(\lambda^*)-\mathbb{G}(\lambda(t))=\mathcal{O}(1/t^{\alpha-1})$ also obtained. In \cite{AttouchCRF2019}, they consider following dynamic:
\[ \ddot{x}(t)+\frac{\alpha}{t}\dot{x}(t)+\beta(t)\nabla \Phi(x(t))=0\]
for problem \eqref{min_fun}, and showed $\mathcal{O}(1/t^2\beta(t))$ rate of convergence under assumption $t\dot{\beta}(t)\leq (\alpha-3)\beta(t)$.  The general damped inertial gradient system with time scaling can be found in \cite{AttouchBCR2021,Attouchcrf2019,BotC2020}.

\subsubsection{Inertial primal-dual dynamics}
For the  affine  constrained convex optimization problem \eqref{question}, the most popular numerical methods and dynamics are based on the primal-dual framework. In recent years, many first-order dynamical systems were proposed for a better understanding of  iterative  schemes of the  numerical algorithms, (see \cite{AttouchCFR2021,BotL2020,Luo2021,Qu2018}). How to extend the dynamics $(\text{IGS}_{\alpha})$ and $(\text{IGS}_{\alpha,\epsilon})$ to second-order primal-dual dynamics for solving  problem \eqref{question} is a problem worth studying.  Recently, Zeng et al. \cite{Zeng2019}  proposed the following damped primal-dual dynamical system for solving the problem \eqref{question}:
\begin{equation}\label{dy_zeng}
	\begin{cases}
		\ddot{x}(t)+\frac{\alpha}{t}\dot{x}(t) = -\nabla f(x(t))-A^T(\lambda(t)+\delta t\dot{\lambda}(t))-\sigma A^T(Ax(t)-b),\\
		\ddot{\lambda}(t)+\frac{\alpha}{t}\dot{\lambda}(t) = A(x(t)+\delta t\dot{x}(t))-b,
	\end{cases}
\end{equation}
In this dynamic, the  damping coefficients   $\alpha(t)=\frac{\alpha}{t}, \delta(t)=\delta t$. When $\alpha>3$ and $\delta=\frac{1}{2}$, they showed that the trajectory  satisfies the following asymptotic convergence rate
\begin{equation}\label{eq_intr_con}
	\mathcal{L}(x(t),\lambda^*)-\mathcal{L}(x^*,\lambda^*)= \mathcal{O}(1/t^{2}),\quad \|Ax(t)-b\|= \mathcal{O}{(1/t)},
\end{equation}
they also obtained $\mathcal{L}(x(t),\lambda^*)-\mathcal{L}(x^*,\lambda^*)=\mathcal{O}(1/t^{2\alpha/3})$ with $\alpha\leq 3,\ \delta=\frac{3}{2\alpha}$.
He et al. \cite{HeHF2020} and  Attouch et al. \cite{AttouchBCR2021}  extended dynamic \eqref{dy_zeng} to solve separable convex optimization problems with general conditions. The ``second-order" $+$ ``first-order" primal-dual dynamics with time scaling  was investigated by \cite{HeHF2021C,HeHF2021F}.


In the next, by the substitution of variables in dynamic \eqref{dy_zeng}, let's illustrate the role of time scaling $\beta(t)$ in dynamic \eqref{dy_dynamic}. Suppose that $\alpha>3$ and $\delta=\frac{1}{2}$ in \eqref{dy_zeng}, $(x^*,\lambda^*)\in \Omega$. Let's make the change of time variable $t = \upsilon (p)$, where $\upsilon:\mathbb{R}\to\mathbb{R}$ and $\lim_{p\to+\infty}\upsilon (p)=+\infty$.
Set $\bar{x}(p) = x(\upsilon(p))$ and $\bar{\lambda}(p) = \lambda(\upsilon(p))$.
By the chain rule, we have
\[\dot{\bar{x}}(p) =\dot{x}(\upsilon(p))\dot{\upsilon}(p),\quad \ddot{\bar{x}}(p) =\dot{x}(\upsilon(p))\ddot{\upsilon}(p)+\ddot{x}({\upsilon}(p))\dot{\upsilon}(p)^2 \]
and
\[\dot{\bar{\lambda}}(p) =\dot{\lambda}(\upsilon(p))\dot{\upsilon}(p),\quad \ddot{\bar{\lambda}}(p) =\dot{\lambda}(\upsilon(p))\ddot{\upsilon}(p)+\ddot{\lambda}(\dot{\upsilon}(p))\upsilon(p)^2. \]
Then rewritten \eqref{dy_zeng} in terms of $\bar{x}(\cdot),\ \bar{\lambda}(\cdot)$ and its derivatives, we obtain
\begin{equation}\label{dy_time_res}
	\begin{cases}
		\ddot{\bar{x}}(p)+\left(\alpha\frac{\dot{\upsilon}(p)}{{\upsilon}(p)}-\frac{\ddot{\upsilon}(p)}{\dot{\upsilon}(p)}\right)\dot{\bar{x}}(p) = -\dot{\upsilon}(p)^2(\nabla f(\bar{x}(p))+A^T(\bar{\lambda}(p)+\frac{\upsilon(p)}{2\dot{\upsilon}(p)}\dot{\bar{\lambda}}(p))+\sigma A^T(A\bar{x}(p)-b),\\
		\ddot{\bar{\lambda}}(p)+\left(\alpha\frac{\dot{\upsilon}(p)}{{\upsilon}(p)}-\frac{\ddot{\upsilon}(p)}{\dot{\upsilon}(p)}\right)\dot{\bar{\lambda}}(p) = \dot{\upsilon}(p)^2(A(\bar{x}(p)+\frac{\upsilon(p)}{2\dot{\upsilon}(p)}\dot{\bar{x}}(p))-b).
	\end{cases}
\end{equation}
This leads to the time scaling coefficient $\beta(p)= \dot{\upsilon}(p)^2$ and the  damping coefficients
$\alpha(p) = \alpha\frac{\dot{\upsilon}(p)}{{\upsilon}(p)}-\frac{\ddot{\upsilon}(p)}{\dot{\upsilon}(p)},\ \delta(p) = \frac{\upsilon(p)}{2\dot{\upsilon}(p)}.$
The  convergence rate \eqref{eq_intr_con} becomes
\[\mathcal{L}(\bar{x}(p),\lambda^*)-\mathcal{L}(x^*,\lambda^*) =\mathcal{O}(\frac{1}{\upsilon(p) ^2}),\quad\|A\bar{x}(p)-b\| =\mathcal{O}(\frac{1}{\upsilon(p)}).\]

In the next, we investigate two model examples. First, taking $\upsilon(p)=e^p$, then \eqref{dy_time_res} reads
\begin{equation}\label{dy_ex1}
	\begin{cases}
		\ddot{\bar{x}}(p)+(\alpha-1)\dot{\bar{x}}(p) = -e^{2p}(\nabla f(\bar{x}(p))+A^T(\bar{\lambda}(p)+\frac{1}{2}\dot{\bar{\lambda}}(p))+\sigma A^T(A\bar{x}(p)-b)),\\
		\ddot{\bar{\lambda}}(p)+(\alpha-1)\dot{\bar{\lambda}}(p) = e^{2p}(A(\bar{x}(p)+\frac{1}{2}\dot{\bar{x}}(p))-b).
	\end{cases}
\end{equation}
In this case, the damping coefficients   $\alpha(p)= \alpha-1$, $\delta(p)=\frac{1}{2}$ are constants,  the time scaling coefficient is $\beta(p)= e^{2p}$, and the convergence rate becomes
 \[\mathcal{L}(\bar{x}(p),\lambda^*)-\mathcal{L}(x^*,\lambda^*) =\mathcal{O}(\frac{1}{e^{2p}}),\quad \|A\bar{x}(p)-b\| =\mathcal{O}(\frac{1}{e^{p}}).\]
 Taking $\upsilon(p)=p^\kappa$ with $\kappa>0$, then \eqref{dy_time_res} reads
\begin{equation}\label{dy_ex2}
	\begin{cases}
		\ddot{\bar{x}}(p)+\frac{1+(\alpha-1)\kappa}{p}\dot{\bar{x}}(p) = -\kappa^2p^{2(\kappa-1)}(\nabla f(\bar{x}(p))+A^T(\bar{\lambda}(p)+\frac{p}{2\kappa}\dot{\bar{\lambda}}(p))+\sigma A^T(A\bar{x}(p)-b)),\\
		\ddot{\bar{\lambda}}(p)+\frac{1+(\alpha-1)\kappa}{s}\dot{\bar{\lambda}}(p) = \kappa^2p^{2(\kappa-1)}(A(\bar{x}(p)+\frac{p}{2\kappa}\dot{\bar{x}}(p))-b).
	\end{cases}
\end{equation}
the convergence rate becomes
 \[\mathcal{L}(\bar{x}(p),\lambda^*)-\mathcal{L}(x^*,\lambda^*) =\mathcal{O}(\frac{1}{p^{2\kappa}}),\quad \|A\bar{x}(p)-b\| =\mathcal{O}(\frac{1}{t^\kappa}),\]
the damping coefficient $\alpha(p) = \frac{1+(\alpha-1)\kappa}{p}$. For $\kappa \geq 1$, we have $ 1+(\alpha-1)\kappa \geq \alpha$ , so damping coefficient similar to \eqref{dy_zeng}, where $\alpha(t) =\frac{\alpha}{t}$.

\subsection{Organisation}
In Section 2, we present the rate of convergence in the different choice of damping coefficient and extrapolation coefficient under the suitable  assumptions on time scaling coefficient and external perturbation. Section 3 concludes the paper. Some technical proofs and lemmas are postponed to Appendix .

\section{Main results}
In this paper, we will investigate the dynamic \eqref{dy_dynamic} with damping coefficient $\alpha(t)=\frac{\alpha}{t^r}$ and  extrapolation coefficient $\delta(t)=\delta t^s$, where $\alpha>0$, $\delta>0$, $0\leq r\leq s\leq 1$. The the dynamic \eqref{dy_dynamic} becomes:
\begin{equation}\label{dy_Appendix}
	\begin{cases}
		\ddot{x}(t)+\frac{\alpha}{t^r}\dot{x}(t) = -\beta(t)(\nabla f(x(t))+A^T(\lambda(t)+\delta t^s\dot{\lambda}(t))+\sigma A^T(Ax(t)-b))+\epsilon(t),\\
		\ddot{\lambda}(t)+\frac{\alpha}{t^r}\dot{\lambda}(t) = \beta(t)(A(x(t)+\delta t^s\dot{x}(t))-b).
	\end{cases}
\end{equation}
Before investigating the rate of convergence, we first discuss the existence and uniqueness of solutions for dynamical system \eqref{dy_Appendix}.

When $\nabla f(x)$ is Lipschitz continuous on $\mathcal{H}_1$, from \cite[Theorem 4.2]{AttouchBCR2021}, for any $(x_0,\lambda_0,u_0,v_0)$, the dynamic \eqref{dy_Appendix} has a unique strong global  solution $(x(t),\lambda(t))$, in which (i): $x(t)\in\mathcal{C}^2([t_0,+\infty),\mathcal{H}_1)$,  $\lambda(t)\in\mathcal{C}^2([t_0,\infty),\mathcal{H}_2)$; (2): $(x(t),\lambda(t))$ and $(\dot{x}(t),\dot{\lambda}(t))$ are locally absolutely continuous; (3): for almost every $t\in [0,+\infty)$, \eqref{dy_Appendix} holds, and $(x(t_0),\lambda(t_0))=(x_0,\lambda_0)$ and $(\dot{x}(t_0),\dot{\lambda}(t_0))=(u_0,v_0)$.

When $\nabla f(x)$ is locally Lipschitz continuous on $\mathcal{H}_1$, following from the Picard-Lindelof Theorem (see \cite[Theorem 2.2]{Teschl2012}), we can establish the local existence and uniqueness solution of dynamic \eqref{dy_Appendix} as follows:
\begin{proposition}\label{pro_local_exist}
	Let $f$  be  continuously differentiable function such that $\nabla f$ is locally Lipschitz continuous, $\beta:[t_0,+\infty)\to (0,+\infty)$ be a continuous function, $\epsilon:[t_0,+\infty)\to\mathcal{H}_1$ be locally integrable. Then for any $(x_0,\lambda_0,u_0,v_0)$, there exists a unique solution $(x(t),\lambda(t))$ with $x(t)\in\mathcal{C}^2([t_0,T),\mathcal{H}_1)$,  $\lambda(t)\in\mathcal{C}^2([t_0,T),\mathcal{H}_2)$ of the dynamic \eqref{dy_Appendix} satisfying $(x(t_0),\lambda(t_0))=(x_0,\lambda_0)$ and $(\dot{x}(t_0),\dot{\lambda}(t_0))=(u_0,v_0)$ on a maximal interval $[t_0,T)\subseteq[t_0,+\infty)$.
\end{proposition}
So under the assumptions in Proposition \ref{pro_local_exist}, we obtain that there exists a unique solution $(x(t),\lambda(t))$ defined on maximal interval $[t_0,T)\subseteq[t_0,+\infty)$. If we can prove that the derivative of trajectory $(\dot{x}(t),\dot{\lambda}(t))$ is bounded on $[t_0,T)$, it follows from assumptions that $(\ddot{x}(t),\ddot{\lambda}(t))$ is also bounded on $[t_0,T)$. This implies that $(x(t),\lambda(t))$ and its derivative $(\dot{x}(t),\dot{\lambda}(t))$ have a limit at $t = T$, and therefore can be continued, a contradiction. Thus $T=+\infty$, we obtain the existence and uniqueness of global solution of dynamic \eqref{dy_Appendix}.
To simplify the proof process, we assume that the global solution of dynamic \eqref{dy_dynamic} exists.  We will discuss the existence and uniqueness of global solution of dynamics \eqref{dy_Appendix} in the case $r=0,s\in[0,1]$ later, and it can be proved similarly for other cases.

In order to  investigate the convergence rates of dynamic \eqref{dy_Appendix} under different choices of $r, s$.  We  construct the different energy functions, fixed $(x^*,\lambda^*)\in\Omega$, for any $\lambda\in\mathcal{H}_2$, define the energy function $\mathcal{E}^{\lambda,\rho}_{\epsilon}:[t_0,+\infty)\to\mathbb{R}$ as
\begin{eqnarray}\label{eq_A_ex}
	\mathcal{E}^{\lambda,\rho}_{\epsilon}(t)=\mathcal{E}^{\lambda,\rho}(t)-\int^{t}_{t_0}\langle \theta(w)(x(w)-x^*)+w^\rho\dot{x}(w),w^\rho \epsilon(w)\rangle dw,
\end{eqnarray}
where
\begin{equation}\label{eq_A1}
	\mathcal{E}^{\lambda,\rho}(t) = \mathcal{E}_0(t)+\mathcal{E}_1(t)+\mathcal{E}_2(t),
\end{equation}
with
\begin{equation*}\label{eq_A2}
	\begin{cases}
		\mathcal{E}_0(t) = t^{2\rho}\beta(t)(\mathcal{L}^{\sigma}(x(t),\lambda)-\mathcal{L}^{\sigma}(x^*,\lambda)	),\\
		\mathcal{E}_1(t) = \frac{1}{2}\|\theta(t)(x(t)-x^*)+t^\rho\dot{x}(t)\|^2+\frac{\eta(t)}{2}\|x(t)-x^*\|^2,\\
		\mathcal{E}_2(t) = \frac{1}{2}\|\theta(t)(\lambda(t)-\lambda)+t^\rho\dot{\lambda}(t)\|^2+\frac{\eta(t)}{2}\|\lambda(t)-\lambda\|^2,
	\end{cases}
\end{equation*}
$\theta,\eta:[t_0,+\infty)\to \mathbb{R}$ are two smooth functions, and $\rho\geq 0$.

The key point of our proof is to find the appropriate $\theta(t),\eta(t)$ to ensure that the energy function $\mathcal{E}^{\lambda,\rho}_{\epsilon}(t)$ is decreasing. To avoid repeated calculations, we  list the main calculation procedures in \ref{APPENDIX_A1}.

\subsection{Case $r=0,\ s\in[0,1]$}
Let us first consider the case when $r=0$, $s\in[0,1]$, i.e., the dynamic \eqref{dy_Appendix}:
\begin{equation}\label{dy_dy1}
	\begin{cases}
		\ddot{x}(t)+\alpha\dot{x}(t) = -\beta(t)(\nabla f(x(t))+A^T(\lambda(t)+\delta t^s \dot{\lambda}(t))+\sigma A^T(Ax(t)-b))+\epsilon(t) ,\\
		\ddot{\lambda}(t)+\alpha\dot{\lambda}(t) = \beta(t)(A(x(t)+\delta t^s\dot{x}(t))-b),
	\end{cases}
\end{equation}
with $\alpha>0,\ \delta>0,\ \sigma\geq 0,\ t\geq t_0> 0$.

\begin{theorem}\label{th_th1}
	Assume that $\beta:[t_0,+\infty)\to(0,+\infty)$ is continuous differentiable function with
	\begin{equation}\label{ass_th1}
		t^s\dot{\beta}(t)\leq (\frac{1}{\delta}-st^{s-1})\beta(t),
	\end{equation}
and $\epsilon:[t_0,+\infty)\to\mathcal{H}_1$ is a  integrable function with
\begin{equation*}
\int^{+\infty}_{t_0}t^{\frac{s}{2}}\|\epsilon(t)\|dt<+\infty.
\end{equation*}
Suppose $\alpha\delta>1$ when $s=0$; $\delta\leq 1$ when $s=1$, and $\sigma>0$.  Let $(x(t),\lambda(t))$ be a global solution of the dynamic \eqref{dy_dy1} and $(x^*,\lambda^*)\in\Omega$. Then   $(x(t),\lambda(t))$ is bounded, and the following conclusions hold:
\begin{itemize}
	\item [(i)] $\int^{+\infty}_{t_0}((\frac{1}{\delta}-st^{s-1})\beta(t)-t^s\dot{\beta}(t))(\mathcal{L}^{\sigma}(x(t),\lambda^*)-\mathcal{L}^{\sigma}(x^*,\lambda^*)) dt <+\infty$.
	\item [(ii)] $\int^{+\infty}_{t_0}t^{s}(\|\dot{x}(t)\|^2+\|\dot{\lambda}(t)\|^2)dt <+\infty$,\ $\int^{+\infty}_{t_0}\beta(t)\|Ax(t)-b\|^2 dt <+\infty$.
	\item [(iii)]  $\|\dot{x}(t)\|+\|\dot{\lambda}(t)\| =\mathcal{O}(\frac{1}{t^{s/2}})$.
	\item [(iv)] When $ \lim_{t\to+\infty}t^s\beta(t) = +\infty$:
	 \[ \mathcal{L}(x(t),\lambda^*)-\mathcal{L}(x^*,\lambda^*) =\mathcal{O}(\frac{1}{t^{s}\beta(t)}),\quad\|Ax(t)-b\| =\mathcal{O}(\frac{1}{t^{s/2}\sqrt{\beta(t)}}).\]
\end{itemize}
\end{theorem}

\begin{proof}
Given $\lambda\in\mathcal{H}_2$, define energy functions $\mathcal{E}^{\lambda,\rho}(t)$ and $\mathcal{E}^{\lambda,\rho}_{\epsilon}(t)$  same as  \eqref{eq_A1}, \eqref{eq_A_ex} with $r=0$, $s\in [0,1]$, $\rho=\frac{s}{2}$, and
\begin{equation}\label{eq_th1_0}
	\theta(t) = \frac{1}{\delta}t^{-s/2}, \quad \eta(t) = \frac{1}{\delta}(\alpha-\frac{1}{\delta}t^{-s}).
\end{equation}
 By computations, we can verify that  \eqref{eq_A4} and \eqref{eq_A6} hold.

{\bf Case $s=0$:} $\theta(t) = \frac{1}{\delta}$ and $\eta(t) =\frac{\alpha\delta-1}{2\delta^2}$. Since $\alpha\delta>1$,  we obtain that \eqref{eq_A3}, \eqref{eq_A5} hold, and then  \eqref{eq_A8} holds,
\begin{eqnarray}\label{eq_th1_1}
	 \theta(t)+\rho t^{\rho-1}-\alpha t^{\rho-r} =\frac{1}{\delta}-\alpha<0.
\end{eqnarray}
It follows from \eqref{ass_th1} that
\begin{eqnarray}\label{eq_th1_2}
	t^{\rho}\dot{\beta}(t)+(2\rho t^{\rho-1}-\theta(t))\beta(t)= \dot{\beta}(t)-\frac{1}{\delta}\beta(t)\leq 0
\end{eqnarray}
for all $t\geq t_0$. Taking $\lambda = \lambda^*$, then $\mathcal{L}^{\sigma}(x(t),\lambda^*)- \mathcal{L}^{\sigma}(x^*,\lambda^*)\geq 0$, it follows from \eqref{eq_th1_1}, \eqref{eq_th1_2} and \eqref{eq_A8} that
\begin{eqnarray}\label{eq_th1_3}
\dot{\mathcal{E}}^{\lambda^*,\rho}_{\epsilon}(t) &\leq& (\frac{1}{\delta}-\alpha) (\|\dot{x}(t)\|^2+\|\dot{\lambda}(t)\|^2)- \frac{\sigma\beta(t)}{2\delta}\|Ax(t)-b)\|^2\nonumber\\
&&+(\dot{\beta}(t)-\frac{1}{\delta}\beta(t))(\mathcal{L}^{\sigma}(x(t),\lambda^*)- \mathcal{L}^{\sigma}(x^*,\lambda^*))\\
	&\leq & 0. \nonumber
\end{eqnarray}
So ${\mathcal{E}}^{\lambda^*,\rho}_{\epsilon}(\cdot)$ is nonincreasing on $[t_0,+\infty)$, and then
\begin{equation}\label{eq_th1_4}
	{\mathcal{E}}^{\lambda^*,\rho}_{\epsilon}(t)\leq {\mathcal{E}}^{\lambda^*,\rho}_{\epsilon}(t_0), \quad
	\forall t\geq t_0.
\end{equation}
By the definition of ${\mathcal{E}}^{\lambda^*,\rho}(\cdot)$ and ${\mathcal{E}}^{\lambda^*,\rho}_{\epsilon}(\cdot)$, we have
\begin{equation*}\label{eq_th1_5}
	\frac{1}{2}\|\frac{1}{\delta}(x(t)-x^*)+\dot{x}(t)\|^2\leq {\mathcal{E}}^{\lambda^*,\rho}_{\epsilon}(t_0)+\int^{t}_{t_0}\langle \frac{1}{\delta}(x(w)-x^*)+\dot{x}(w),\epsilon(w)\rangle dw.
\end{equation*}
By  Cauchy-Schwarz inequality, we get
\begin{equation*}\label{eq_th1_6}
	\frac{1}{2}\|\frac{1}{\delta}(x(t)-x^*)+\dot{x}(t)\|^2\leq |{\mathcal{E}}^{\lambda^*,\rho}_{\epsilon}(t_0)|+\int^{t}_{t_0} \|\frac{1}{\delta}(x(w)-x^*)+\dot{x}(w)\|\|\epsilon(w)\| dw,
\end{equation*}
then applying Lemma  \ref{le_A1} with $\mu(t) = \|\frac{1}{\delta}(x(t)-x^*)+\dot{x}(t)\|$, we obtain
\begin{equation}\label{eq_th1_7}
	\sup_{t\geq t_0} \|\frac{1}{\delta}(x(t)-x^*)+\dot{x}(t)\|\leq \sqrt{2|{\mathcal{E}}^{\lambda^*,\rho}_{\epsilon}(t_0)|}+\int^{+\infty}_{t_0}\|\epsilon(t)\|dt<+\infty.
\end{equation}
 It is easy to verify ${\mathcal{E}}^{\lambda^*,\rho}(t)\geq 0$ for all $t\geq t_0$, then we have
\begin{equation*}\label{eq_th1_8}
	\inf_{t\geq t_0 }{\mathcal{E}}^{\lambda^*,\rho}_{\epsilon}(t) \geq -\sup_{t\geq t_0}\|\frac{1}{\delta}(x(t)-x^*)+\dot{x}(t)\|\times\int^{+\infty}_{t_0}\|\epsilon(s)\| ds > -\infty
\end{equation*}
and
\begin{equation*}\label{eq_th1_9}
	\sup_{t\geq t_0 }{\mathcal{E}}^{\lambda^*,\rho}(t) \leq{\mathcal{E}}^{\lambda^*,\rho}_{\epsilon}(t_0)+\sup_{t\geq t_0}\|\frac{1}{\delta}(x(t)-x^*)+\dot{x}(t)\|\times\int^{+\infty}_{t_0}\|\epsilon(s)\| ds < +\infty.
\end{equation*}
This together with \eqref{eq_th1_4} and the definition of $\mathcal{E}^{\lambda^*,\rho}(\cdot)$ yields the boundedness of $\mathcal{E}^{\lambda^*,\rho}(\cdot)$ and ${\mathcal{E}}^{\lambda^*,\rho}_{\epsilon}(\cdot)$.
 By integrating inequality \eqref{eq_th1_3} on $[t_0.+\infty)$, it follows the boundedness of $\mathcal{E}^{\lambda^*,\rho}_{\epsilon}(\cdot)$  that
\begin{eqnarray*}\label{eq_th1_10}
	&&(\alpha-\frac{1}{\delta})\int^{+\infty}_{t_0}\|\dot{x}(t)\|^2+\|\dot{\lambda}(t)\|^2 dt +\int^{+\infty}_{t_0}(\frac{1}{\delta}\beta(t)-\dot{\beta}(t))(\mathcal{L}^{\sigma}(x(t),\lambda^*)- \mathcal{L}^{\sigma}(x^*,\lambda^*))dt\\
	&&\qquad\qquad\qquad +\frac{\sigma}{2\delta}\int^{+\infty}_{t_0} \beta(t)\|Ax(t)-b)\|^2\\
	&&\qquad\qquad\quad<+\infty.
\end{eqnarray*}
This together with $\frac{1}{\delta}<\alpha$ yields $(i)-(ii)$.

 Since $\eta(t)=\frac{\alpha\delta-1}{2\delta^2}>0$, By the boundedness of $\mathcal{E}^{\lambda^*,\rho}(\cdot)$, we obtain that $\|x(t)-x^*\|^2,\|\lambda(t)-\lambda^*\|^2$, $ \|\frac{1}{\delta}(x(t)-x^*)+\dot{x}(t)\|$ and $ \|\frac{1}{\delta}(\lambda(t)-\lambda^*)+\dot{\lambda}(t)\|$  are bounded, and then the trajectory $(x(t),\lambda(t))$ is bounded,
	\[\sup_{t_0\in[t_0,+\infty)} \|\dot{x}(t)\|\leq \frac{1}{\delta} \sup_{t\in[t_0,+\infty)}\|x(t)-x^*\|+\sup_{t\in[t_0,+\infty)}\|\frac{1}{\delta}(x(t)-x^*)+\dot{x}(t)\| <+\infty,\]
similarly, $\sup_{t_0\in[t_0,+\infty)} \|\dot{\lambda}(t)\|<+\infty$,
this is $(iii)$. When $\lim_{t\to+\infty}\beta(t)=+\infty$, following from the boundedness of $\mathcal{E}^{\lambda^*,\rho}(\cdot)$, we get
\begin{equation*}\label{eq_th1_11}
	\mathcal{L}^{\sigma}(x(t),\lambda^*)-\mathcal{L}^{\sigma}(x^*,\lambda^*) =\mathcal{O}(\frac{1}{\beta(t)}).	
\end{equation*}
Since $\mathcal{L}^{\sigma}(x(t),\lambda^*)-\mathcal{L}^{\sigma}(x^*,\lambda^*)=\mathcal{L}(x(t),\lambda^*)-\mathcal{L}(x^*,\lambda^*)+\frac{\sigma}{2}\|Ax(t)-b\|^2$, then we obtain $(iv)$.

 {\bf Case $s\in(0,1]$:} There exists $t_1\geq t_0$ such that
\begin{eqnarray}\label{eq_th1_12}
\frac{1}{\delta}t^{-s}+\frac{s}{2} t^{-1}\leq \frac{\alpha}{2}, \quad\forall\ t\geq t_1,
\end{eqnarray}
this together with \eqref{eq_th1_0} yields
\begin{equation}\label{eq_th1_13}
	 \eta(t)\geq \frac{\alpha}{2\delta}> 0, \quad \forall\ t\geq t_1.
\end{equation}
We can  compute that
\begin{eqnarray*}\label{eq_th1_14}
	\theta(t)\dot{\theta}(t)+\frac{\dot{\eta}(t)}{2} = 0.
\end{eqnarray*}
Then \eqref{eq_A3}-\eqref{eq_A6} are satisfied  for any $t\geq t_1$.
It follows from \eqref{eq_th1_12} that
\begin{eqnarray}\label{eq_th1_15}
	\theta(t)+\rho t^{\rho-1}-\alpha t^{\rho-r} = t^{s/2}(\frac{1}{\delta}t^{-s}+\frac{s}{2} t^{-1}-\alpha)\leq -\frac{\alpha}{2}t^{s/2}, \quad \forall\ t\geq t_1.
\end{eqnarray}
By computation,
\begin{eqnarray}\label{eq_th1_16}
	t^{\rho}(t^{\rho}\dot{\beta}(t)+(2\rho t^{\rho-1}-\theta(t))\beta(t))= t^s\dot{\beta}(t)-(\frac{1}{\delta}-st^{s-1})\beta(t)\leq 0,\quad\forall \ t\geq t_0.
\end{eqnarray}
Let $\lambda=\lambda^*$, $\mathcal{L}^{\sigma}(x(t),\lambda^*)- \mathcal{L}^{\sigma}(x^*,\lambda^*)\geq 0$. Combining \eqref{eq_th1_15}, \eqref{eq_th1_16} and \eqref{eq_A8}, we get
\begin{eqnarray}\label{eq_th1_17}
	\dot{\mathcal{E}}^{\lambda^*,\rho}_{\epsilon}(t)&\leq& -\frac{\alpha}{2} t^{s}(\|\dot{x}(t)\|^2+\|\dot{\lambda}(t)\|^2)- \frac{\sigma \beta(t)}{2\delta}\|Ax(t)-b)\|^2\nonumber\\
	&& +(t^s\dot{\beta}(t)-(\frac{1}{\delta}-st^{s-1})\beta(t))(\mathcal{L}^{\sigma}(x(t),\lambda^*)- \mathcal{L}^{\sigma}(x^*,\lambda^*)) \\	
	&\leq& 0\nonumber
\end{eqnarray}
for all $t\geq t_1$.  ${\mathcal{E}}^{\lambda^*,\rho}_{\epsilon}(\cdot)$ is nonincreasing on $[t_1,+\infty)$,
\begin{equation*}\label{eq_th1_18}
	{\mathcal{E}}^{\lambda^*,\rho}_{\epsilon}(t)\leq{\mathcal{E}}^{\lambda^*,\rho}_{\epsilon}(t_1),\quad \forall\ t\geq t_1.
\end{equation*}
By the definition of ${\mathcal{E}}^{\lambda^*,\rho}(\cdot)$ and ${\mathcal{E}}^{\lambda^*,\rho}_{\epsilon}(\cdot)$, for all $t\geq t_1$ we have
\begin{equation*}\label{eq_th1_19}
	\frac{1}{2}\|\frac{1}{\delta}t^{-s/2}(x(t)-x^*)+t^{s/2}\dot{x}(t)\|^2\leq {\mathcal{E}}^{\lambda^*,\rho}_{\epsilon}(t_1)+ \int^t_{t_1}\langle\frac{1}{\delta}w^{-s/2}(x(w)-x^*)+w^{s/2}\dot{x}(w),w^{s/2} \epsilon(w)\rangle dw.
\end{equation*}
By similar arguments in {\bf Case s=0}, we obtain the boundedness of ${\mathcal{E}}^{\lambda^*,\rho}(\cdot)$ and ${\mathcal{E}}^{\lambda^*,\rho}_{\epsilon}(\cdot)$. Integrating inequality \eqref{eq_th1_17} on $[t_1,+\infty)$, we get the results $(i)-(ii)$.

Since ${\mathcal{E}}^{\lambda^*,\rho}(\cdot)$ is bounded, following from the definition of ${\mathcal{E}}^{\lambda^*,\rho}(\cdot)$, we obtain $(iv)$,
\[\sup_{t\geq t_0} \eta(t)\|x(t)-x^*\|^2<+\infty\]
and
\[\sup_{t\geq t_0} \|\frac{1}{\delta}t^{-s/2}(x(t)-x^*)+t^{s/2}\dot{x}(t)\|<+\infty.\]
This together with \eqref{eq_th1_13} and $s\in(0,1]$ implies
\[\sup_{t\geq t_0} \|x(t)-x^*\|<+\infty\]
and
\begin{eqnarray*}\label{eq_th1_20}
	 \sup_{t\geq t_0} t^{s/2}\|\dot{x}(t)\|&\leq&   \frac{1}{\delta}\sup_{t\geq t_0}t^{-s/2}\|x(t)-x^*\|+\sup_{t\geq t_0}\|\frac{1}{\delta}t^{-s/2}(x(t)-x^*)+t^{s/2}\dot{x}(t)\|\\
	& \leq&  \frac{1}{\delta t_0^{s/2}}\sup_{t\geq t_0}\|x(t)-x^*\|+\sup_{t\geq t_0} \|\frac{1}{\delta}t^{-s/2}(x(t)-x^*)+t^{s/2}\dot{x}(t)\|\\
	& <& +\infty.
\end{eqnarray*}
Similarly we have $\sup_{t\geq t_0} \|\lambda(t)-\lambda^*\|<+\infty$, $\sup_{t\geq t_0}t^{s/2}\|\dot{\lambda}(t)\|<+\infty$.  Then we obtain the boundedness of $(x(t),\lambda(t))$  and $(iii)$.
\end{proof}

\begin{remark}\label{re_th1_1}
	From Proposition \ref{pro_local_exist}, there exists a unique local solution $(x(t),$ $\lambda(t))$ of the  dynamic \eqref{dy_dy1} defined on a maximal interval $[t_0,T)$ with $T\leq +\infty$. If we pick a appropriate  $t_0>0$,  following from the proof process in Theorem \ref{th_th1} and $(iii)$, we can obtain $\sup_{t\in[t_0,T)}\|\dot{x}(t)\|+\|\dot{\lambda}(t)\|<+\infty$, and then  $T=+\infty$, the  existence and uniqueness of  global solution of the dynamic \eqref{dy_dy1} is established.
\end{remark}

\begin{remark}\label{re_th1_2}
	From Theorem \ref{th_th1}, we can see that for same damping $\alpha(t)=\alpha$, choosing another damping  $\delta(t)= \frac{\delta}{t^s}$ different,  the different rates of convergence can be obtained. Taking $A=0$, $b=0$, we can obtain the $\mathcal{O}(1/t^s\beta(t))$ convergence rate for dynamic \eqref{dy_heavy} under the assumption $t^s\dot{\beta}(t)\leq (\frac{1}{\delta}-st^{s-1})\beta(t)$ with  $\delta>0$, so Theorem \ref{th_th1}  complements the results  in \cite{Balhag2020}. The assumption $\int^{+\infty}_{t_0}\|\epsilon(t)\|dt<+\infty$ for perturbation $\epsilon(t)$ has been used in \cite{HarauxJ2012}  for asymptotic analysis of heavy ball dynamic.	
\end{remark}

\begin{remark}\label{re_th1_3}
	When $s=0$, choosing $\beta(t)\equiv 1$, then \eqref{ass_th1} is automatically  satisfied. Then from $(i)$, we have $\int^{+\infty}_{t_0} \mathcal{L}^{\sigma}(x(t),\lambda^*)-\mathcal{L}^{\sigma}(x^*,\lambda^*)dt <+\infty$.
Since $ \mathcal{L}^{\sigma}(\cdot,\lambda^*)$ is a convex function with respect to first variable, taking $\bar{x}(t)=\frac{\int^t_{t_0}x(s)ds}{t-t_0}$, we have
\begin{eqnarray*}
	\mathcal{L}^{\sigma}(\bar{x}(t),\lambda^*)-\mathcal{L}^{\sigma}(x^*,\lambda^*) &\leq& \frac{1}{t-t_0}\int^{t}_{t_0}\mathcal{L}^{\sigma}(x(s),\lambda^*)-\mathcal{L}^{\sigma}(x^*,\lambda^*)ds\\
	& \leq& \frac{1}{t-t_0}\int^{+\infty}_{t_0}\mathcal{L}^{\sigma}(x(s),\lambda^*)-\mathcal{L}^{\sigma}(x^*,\lambda^*)ds.
\end{eqnarray*}
Following from the definition of $\mathcal{L}^{\sigma}(\bar{x}(t),\lambda^*)$, we obtain $\mathcal{L}(\bar{x}(t),\lambda^*)-\mathcal{L}(x^*,\lambda^*)=\mathcal{O}(1/t)$ and $\|A\bar{x}(t)-b\|=\mathcal{O}(1/\sqrt{t})$, the $\mathcal{O}(1/t)$ ergodic convergence rate  corresponds to the  convergence rate of the  discrete heavy ball algorithm in \cite{Ghadimi2015}; for general $\beta(t)$ with $s=0$, the similar convergence rate results can be found in \cite{HeHF2021F}. When $s=1$, choosing $\beta(t)\equiv 1$ and $\delta\leq 1$, the $\mathcal{O}(1/t)$ rate of convergence also was investigated in \cite[Theorem 4.4]{HeHF2020} with $r=0$ for problem \eqref{question}, and it is  consistent with results of heavy ball dynamic and algorithm in \cite{Sun2019} for problem \eqref{min_fun}.
\end{remark}

In Theorem \ref{th_th1}, when $\lim_{t\to+\infty}t^s\beta(t)=+\infty$, we show the $\mathcal{O}(1/t^s\beta(t))$  convergence rate of Lagrangian function and $\mathcal{O}(1/t^{s/2}\sqrt{\beta(t)})$  convergence rate of constraint, then
\[|f({x}(t))-f(x^*)|\leq \mathcal{L}({x}(t),\lambda^*)-\mathcal{L}^{\sigma}(x^*,\lambda^*) + \|\lambda^*\|\|A{x}(t)-b\|=\mathcal{O}\left(\frac{1}{t^{s/2}\sqrt{\beta(t)}}\right).\]
 We only can obtain the  $\mathcal{O}(1/t^{s/2}\sqrt{\beta(t)})$  convergence rate of  objection function.

 In the next, we  will investigate the best convergence rates  of objection function and constrain for suitable $\beta(t)$.  When $s=0$, let $\dot{\beta}(t)=\frac{1}{\delta}\beta(t)$. then $\frac{\dot{\beta}(t)}{\beta(t)}=\frac{1}{\delta}$, integrating it on $[t_0,t]$, we have
 \[ \beta(t) = 	\frac{\beta(t_0)}{e^{t_0/\delta}}e^{\frac{t}{\delta}}. \]
In this case, from Theorem \ref{th_th1}, we can obtain the $\mathcal{O}(\frac{1}{e^{\frac{t}{2\delta}}})$ convergence rate of objective function and constraint. Let $\beta(t)=\mu e^{t/\delta}$ with $\mu>0$, we list the following improved convergence rate results, which also can be found in \cite[Proposition 6.2]{AttouchBCR2021} with $\epsilon(t)=0$.

 \begin{theorem}\label{th_th1_1}
	Let $\beta(t)=\mu e^{{t}/{\delta}}$ with $\mu>0$, $\alpha\delta>1$, $s=0$, $\sigma\geq 0$. Assume $\int^{+\infty}_{t_0}\|\epsilon(t)\|dt<+\infty$.  Let  $(x(t),\lambda(t))$ be a solution of dynamic \eqref{dy_dy1} and  $(x^*,\lambda^*)\in\Omega$. Then:
\[ |f(x(t))-f(x^*)| =\mathcal{O}(\frac{1}{ e^{{t}/{\delta}}}),\quad \|Ax(t)-b\| =\mathcal{O}(\frac{1}{{e^{{t}/{\delta}}}}).\]
\end{theorem}
\begin{proof}
Given $\lambda\in\mathcal{H}_2$, recall the energy  functions $\mathcal{E}^{\lambda,\rho}(t)$ and $\mathcal{E}^{\lambda,\rho}_{\epsilon}(t)$  from Theorem \ref{th_th1} with $\beta(t)=\mu e^{{t}/{\delta}},\ s=\rho=0$. Then
\[ t^{\rho}\dot{\beta}(t)+(2\rho t^{\rho-1}-\theta(t))\beta(t) =0,\]
this together with \eqref{eq_th1_1} and \eqref{eq_A8} yields
\begin{eqnarray}\label{eq_th1_1_1}
	\dot{\mathcal{E}}^{\lambda,\rho}_{\epsilon}(t)\leq (\frac{1}{\delta}-\alpha)(\|\dot{x}(t)\|^2+\|\dot{\lambda}(t)\|^2)- \frac{\sigma \beta(t)}{2\delta}\|Ax(t)-b)\|^2 \leq 0, \quad \forall\ t\geq t_0, \lambda\in\mathcal{H}_2.
\end{eqnarray}
So for any $\lambda\in\mathcal{H}_2$, ${\mathcal{E}}^{\lambda,\rho}_{\epsilon}(\cdot)$ is nonincreasing on $[t_0,+\infty)$ such that,
\[{\mathcal{E}}^{\lambda,\rho}_{\epsilon}(t)\leq  {\mathcal{E}}^{\lambda,\rho}_{\epsilon}(t_0),\quad \forall\ t\geq t_0.\]
By the definition of ${\mathcal{E}}^{\lambda,\rho}_{\epsilon}(\cdot)$ and $\sigma\geq 0$, we have
\[f(x(t))-f(x^*)+\langle \lambda, Ax(t)-b\rangle \leq \frac{1}{\mu e^{{t}/{\delta}}}\left(\mathcal{E}^{\lambda,\epsilon}(t_0)+\sup_{t\geq t_0}\|\frac{1}{\delta}(x(t)-x^*)+\dot{x}(t)\|\int^{+\infty}_{t_0}\|\epsilon(t)\|dt\right)\]
for any $\lambda\in\mathcal{H}_1$ and $t\geq t_0$. Taking $\varrho>\|\lambda^*\|$, it follows from Lemma \ref{le_A2} that
\begin{eqnarray}\label{eq_th1_1_2}
	f(x(t))-f(x^*)+\varrho\| Ax(t)-b\| \leq \frac{1}{\mu e^{{t}/{\delta}}}\left(\sup_{\|\lambda\|\leq \varrho}\mathcal{E}^{\lambda,\epsilon}(t_0)+\sup_{t\geq t_0}\|\frac{1}{\delta}(x(t)-x^*)+\dot{x}(t)\|\int^{+\infty}_{t_0}\|\epsilon(t)\|dt\right).
\end{eqnarray}
Denote $C=\sup_{\|\lambda\|\leq \varrho}\mathcal{E}^{\lambda,\epsilon}(t_0)+\sup_{t\geq t_0}\|\frac{1}{\delta}(x(t)-x^*)+\dot{x}(t)\|\int^{+\infty}_{t_0}\|\epsilon(t)\|dt$. Since $\varrho>\|\lambda^*\|$,  $\sup_{\|\lambda\|\leq \varrho}\mathcal{E}^{\lambda,\epsilon}(t_0)\geq \mathcal{E}^{\lambda^*,\epsilon}(t_0)\geq 0$, this together with \eqref{eq_th1_7} yields  $0\leq C<+\infty$.
Following from \eqref{saddle_point}, we have
 \[ f(x(t))-f(x^*)\geq -\|\lambda^*\|\|Ax(t)-b\|,\]
this together with \eqref{eq_th1_1_2} implies
 \[\|Ax(t)-b\|\leq \frac{C}{\mu(\varrho-\lambda^*) e^{{t}/{\delta}}} \]
 and then
 \begin{eqnarray*}
 \frac{-\|\lambda^*\|C}{\mu(\varrho-\lambda^*) e^{{t}/{\delta}}} \leq f(x(t))-f(x^*) \leq \frac{C}{\mu e^{{t}/{\delta}}}.
  \end{eqnarray*}
 We  obtain results from above inequalities.
\end{proof}

\begin{remark}\label{re_th11_1}
	When $s=0$ and $\beta(t)=\mu e^{{t}/{\delta}}$, Theorem \ref{th_th1} obtains $\mathcal{O}(\frac{1}{e^{{t}/{2\delta}}})$ convergence rate of objective function and constraint, it is consistent with convergence rates of dynamic \eqref{dy_ex1}, which is derived from dynamic \eqref{dy_zeng}.  Theorem \ref{th_th1_1} shows that the rate of convergence is actually $\mathcal{O}(\frac{1}{e^{{t}/{\delta}}})$. Then we can obtain the linear convergence rate of dynamic \eqref{dy_dy1} merely under the convexity assumption of $f$, and in this case we also allow the penalty parameter $\sigma$ of augmented Lagrangian function to be zero, which is different in Theorem \ref{th_th1}.
\end{remark}

When $s\in(0,1)$, let $t^s\dot{\beta}(t)= (\frac{1}{\delta}-s t^{s-1})\beta(t)$.
It leads
 \[\beta(t) = \frac{t_0^s\beta(t_0)}{e^{\frac{1}{\delta(1-s)}t_0^{1-s}}}\frac{e^{\frac{1}{\delta(1-s)}t^{1-s}}}{t^s}.\]
 Take $\beta(t) =\mu \frac{e^{\frac{1}{\delta(1-s)}t^{1-s}}}{t^s}$ with $\mu>0$. We investigate the following optimal results.

\begin{theorem}\label{th_th1_2}
	Let $\beta(t) =\mu \frac{e^{\frac{1}{\delta(1-s)}t^{1-s}}}{t^s}$ with $\mu>0,\ s\in(0,1),\ \sigma\geq 0$ . Suppose $\int^{+\infty}_{t_0}t^{s/2}\|\epsilon(t)\|dt<+\infty$. Let   $(x(t),\lambda(t))$ be a solution of dynamic \eqref{dy_dy1} and $(x^*,\lambda^*)\in\Omega$. Then
\[ |f(x(t))-f(x^*)| =\mathcal{O}\left(\frac{1}{e^{\frac{1}{\delta(1-s)}t^{1-s}}}\right),\quad \|Ax(t)-b\| =\mathcal{O}\left(\frac{1}{e^{\frac{1}{\delta(1-s)}t^{1-s}}}\right).\]
\end{theorem}

\begin{proof}
Given $\lambda\in\mathcal{H}_2$, recall the energy  functions $\mathcal{E}^{\lambda,\rho}(t)$ and $\mathcal{E}^{\lambda,\rho}_{\epsilon}(t)$  from Theorem \ref{th_th1} with $\beta(t)=\mu \frac{e^{\frac{1}{\delta(1-s)}t^{1-s}}}{t^s},\ s\in(0,1),\ \rho=\frac{s}{2}$. Then
\[ t^{\rho}\dot{\beta}(t)+(2\rho t^{\rho-1}-\theta(t))\beta(t) =0,\]
this together with \eqref{eq_th1_1} and \eqref{eq_A8} yields
\begin{eqnarray}\label{eq_th2_9}
	\dot{\mathcal{E}}^{\lambda,\rho}_{\epsilon}(t)\leq -\frac{\alpha}{2} t^{s}(\|\dot{x}(t)\|^2+\|\dot{\lambda}(t)\|^2)- \frac{\sigma \beta(t)}{2\delta}\|Ax(t)-b)\|^2 \leq 0, \quad \forall t\geq t_1, \lambda\in\mathcal{H}_2,
\end{eqnarray}
for some $t_1\geq t_0$. By similar arguments in Theorem \ref{th_th1}, we obtain the results.
\end{proof}

When $s=1$, let $t\dot{\beta}(t)= (\frac{1}{\delta}-1)\beta(t)$.
It leads
 \[\beta(t) = \frac{\beta(t_0)}{t_0^{\frac{1}{\delta}-1}}t^{\frac{1}{\delta}-1}.\]
Taking $\beta(t) =\mu t^{\frac{1}{\delta}-1}$ with $\mu>0$. By similar arguments in Theorem \ref{th_th1_1} and Theorem \ref{th_th1_2}, we obtain the following results.

\begin{theorem}\label{th_th1_3}
Let $\beta(t)= \mu t^{\frac{1}{\delta}-1}$ with $\mu>0$, $\delta\leq 1$, $s=1$, $\sigma\geq 0$. Suppose $\int^{+\infty}_{t_0}t^{1/2}\|\epsilon(t)\|dt<+\infty$. Let   $(x(t),\lambda(t))$ be a solution of dynamic \eqref{dy_dy1} and $(x^*,\lambda^*)\in\Omega$. We have
\[ |f(x(t))-f(x^*)| =\mathcal{O}(\frac{1}{t^{1/\delta}}),\quad \|Ax(t)-b\| =\mathcal{O}(\frac{1}{t^{1/\delta}}).\]
\end{theorem}

\begin{remark}\label{re_th13_1}
When $s=1$, taking $\delta=1$ and $\beta(t)\equiv 1$, from Theorem \ref{th_th1_3}, we obtain $\mathcal{O}(\frac{1}{t})$ convergence rates of objective function and constraint, which improves results in Theorem \ref{th_th1} with  time scaling  $\beta(t)\equiv 1$.
\end{remark}

\begin{remark}\label{re_th13_2}
	For damping $\alpha(t) =\alpha$, $\delta=\frac{\delta}{t^s}$ with $s\in[0,1]$, the $\mathcal{O}(1/t^s\beta(t))$ convergence rate in Theorem \ref{th_th1}  shows that convergence results is better as $s$  larger in $[0,1]$. Conversely, following from Theorem \ref{th_th1_1}-Theorem \ref{th_th1_3}, when $s$ is smaller in $[0,1]$, we can obtain  better optimal convergence rates with suitable $\beta(t)$.
\end{remark}

\subsection{Case $r\in(0,1),\ s\in[r,1]$}
In the case $r\in(0,1)$, $s\in[r,1]$, the dynamic \eqref{dy_dynamic} reads:
\begin{equation}\label{dy_dy2}
	\begin{cases}
		\ddot{x}(t)+\frac{\alpha}{t^r}\dot{x}(t) = -\beta(t)(\nabla f(x(t))+A^T(\lambda(t)+\delta t^s\dot{\lambda}(t))+\sigma A^T(Ax(t)-b))+\epsilon(t) ,\\
		\ddot{\lambda}(t)+\frac{\alpha}{t^r}\dot{\lambda}(t) = \beta(t)(A(x(t)+\delta t^s\dot{x}(t))-b).
	\end{cases}
\end{equation}
with $\alpha>0,\ \delta>0,\ \sigma\geq 0,\ t\geq t_0> 0$. We will investigate the convergence properties of dynamic \eqref{dy_dy2}.

\begin{theorem}\label{th_th2}
	Assume that $\beta:[t_0,+\infty)\to(0,+\infty)$ is continuous differentiable function with
\begin{equation}\label{ass_th2}
	t^s\dot{\beta}(t)\leq (\frac{1}{\delta}-\tau t^{s-1})\beta(t)
\end{equation}
and  $\epsilon:[t_0,+\infty)\to\mathcal{H}_1$ satisfies
\begin{equation*}
\int^{+\infty}_{t_0}t^{\tau/2}\|\epsilon(t)\|dt<+\infty,
\end{equation*}
where $\tau\in(0,r+s)$.
Assume $\alpha\delta>1$ when $s=r$;  $\tau\delta\leq 1$ when $s=1$.
Let $(x(t),\lambda(t))$ be a global solution of the dynamic \eqref{dy_dy2} and $(x^*,\lambda^*)\in\Omega$. The following results hold:
\begin{itemize}
\item [(i)] $\int^{+\infty}_{t_0} ((\frac{1}{\delta}t^{\tau-s}-\tau t^{\tau-1})\beta(t)-t^{\tau}\dot{\beta}(t))(\mathcal{L}^{\sigma}(x(t),\lambda^*)-\mathcal{L}^{\sigma}(x^*,\lambda^*)) dt <+\infty$.	
\item [(ii)] $\int^{+\infty}_{t_0}t^{\tau-s}\beta(t)\|Ax(t)-b\|^2 dt <+\infty$,\  $\int^{+\infty}_{t_0}t^{\tau-r}(\|\dot{x}(t)\|^2+\|\dot{\lambda}(t)\|^2)dt <+\infty$.
\item [(iii)]  $\|\dot{x}(t)\|+\|\dot{\lambda}(t)\| =\mathcal{O}(\frac{1}{t^{\tau/2}})$.
\item [(iv)] When $\lim_{t\to+\infty}t^{\tau}\beta(t) = +\infty,$
	\[ \mathcal{L}(x(t),\lambda^*)-\mathcal{L}(x^*,\lambda^*) =\mathcal{O}(\frac{1}{t^{\tau}\beta(t)}), \quad \|Ax(t)-b\| =\mathcal{O}(\frac{1}{t^{\tau/2}\sqrt{\beta(t)}}).\]
\end{itemize}
\end{theorem}

\begin{proof}
Given $\lambda\in\mathcal{H}_2$, 	recall energy functions $\mathcal{E}^{\lambda,\rho}(t)$ and $\mathcal{E}^{\lambda,\rho}_{\epsilon}(t)$  from \eqref{eq_A1}, \eqref{eq_A_ex} with
 $r\in(0,1),s\in[r,1], \rho=\frac{\tau}{2}$ and
\begin{equation}\label{eq_th2_1}
	\theta(t) = \frac{1}{\delta}t^{\tau/2-s},\quad \eta(t) = -\frac{1}{\delta}t^{\tau-s-r}(\frac{1}{\delta}t^{r-s}+(\tau-s)t^{r-1}-\alpha).
\end{equation}
 Then the equations \eqref{eq_A4} and \eqref{eq_A6} are automatically satisfied.

 We claim that there  exists  $C_1<0$ and $t_1\geq t_0$ such that
\begin{eqnarray}\label{eq_th2_2}
\frac{1}{\delta}t^{r-s}+\frac{\tau}{2} t^{r-1} -\alpha\leq C_1, \quad\forall\ t\geq t_1.
\end{eqnarray}
Indeed, when $s=r$, since $\alpha\delta>1$ and $r\in(0,1)$, there exists  $t_1\geq t_0$ such that $\frac{1}{\delta}t^{r-s}+\frac{\tau}{2} t^{r-1} -\alpha=\frac{1}{\delta}-\alpha+\frac{\tau}{2} t^{r-1}\leq \frac{1}{2}(\frac{1}{\delta}-\alpha)<0$; when $s\in(r,1]$, since $r\in(0,1)$, there exists $t_1\geq t_0$ such that
$\frac{1}{\delta}t^{r-s}+\frac{\tau}{2} t^{r-1} -\alpha\leq - \frac{\alpha}{2}<0$. Since $\frac{\tau}{2}<\frac{r+s}{2}\leq s$, it follows from \eqref{eq_th2_2} that
\begin{equation*}\label{eq_th2_3}
	\frac{1}{\delta}t^{r-s}+(\tau-s)t^{r-1}-\alpha\leq C_1, \quad\forall\ t\geq t_1,
\end{equation*}
 and it yields
\begin{equation}\label{eq_th2_4}
	 \eta(t)\geq \frac{-C_1}{\delta}t^{\tau-s-r}\geq 0, \quad \forall\ t\geq t_1.
\end{equation}
Since $\tau\in(0,s+r)$, then there exist $t_2\geq t_1$ such that
\[\alpha(\tau-s-r)-(\tau-s-1)(\tau-s)t^{r-1}<0,\quad \forall \ t\geq t_2,\]
so we can compute
\begin{eqnarray*}\label{eq_th2_5}
	\theta(t)\dot{\theta}(t)+\frac{\dot{\eta}(t)}{2}=\frac{1}{2\delta}t^{\tau-s-r-1}(\alpha(\tau-s-r)-(\tau-s-1)(\tau-s)t^{r-1})<0
\end{eqnarray*}
for all $t\geq t_2$. Then  \eqref{eq_A3} and \eqref{eq_A5} hold for any $t\geq t_2$.

It follows from \eqref{eq_th2_2} that
\begin{eqnarray}\label{eq_th2_6}
	\theta(t)+\rho t^{\frac{\tau}{2}-1}-\alpha t^{\frac{\tau}{2}-r} = t^{\frac{\tau}{2}-r}(\frac{1}{\delta}t^{r-s}+\frac{\tau}{2} t^{r-1}-\alpha)\leq C_1t^{\frac{\tau}{2}-r}<0, \quad \forall\ t\geq t_1.
\end{eqnarray}
By computation, and from \eqref{ass_th2}, we have
\begin{eqnarray}\label{eq_th2_7}
	&&t^{\rho}\dot{\beta}(t)+(2\rho t^{\rho-1}-\theta(t))\beta(t)= t^{\frac{\tau}{2}-s}(t^s\dot{\beta}(t)- (\frac{1}{\delta}-\tau t^{s-1})\beta(t)) \leq 0.\nonumber
\end{eqnarray}
for all $t\geq t_0$. Let $\lambda=\lambda^*$, then $\mathcal{L}^{\sigma}(x(t),\lambda^*)- \mathcal{L}^{\sigma}(x^*,\lambda^*)\geq 0$,  this together with  \eqref{eq_th2_1}, \eqref{eq_th2_6} and \eqref{eq_A8}  yields
\begin{eqnarray}\label{eq_th2_8}
	\dot{\mathcal{E}}^{\lambda^*,\rho}_{\epsilon}(t) &\leq& C_1 t^{\tau-r}(\|\dot{x}(t)\|^2+\|\dot{\lambda}(t)\|^2)- \frac{\sigma t^{\tau-s}\beta(t)}{2\delta}\|Ax(t)-b)\|^2\\	
	&& +(t^{\tau}\dot{\beta}(t)- (\frac{1}{\delta}t^{\tau-s}-\tau t^{\tau-1})\beta(t))(\mathcal{L}^{\sigma}(x(t),\lambda^*)- \mathcal{L}^{\sigma}(x^*,\lambda^*))\nonumber\\
	&\leq& 0\nonumber
\end{eqnarray}
for all $t\geq t_2$. Then ${\mathcal{E}}^{\lambda^*,\rho}_{\epsilon}(\cdot)$ is nonincreasing on $[t_2,+\infty)$,
\begin{equation*}\label{eq_th2_9}
	{\mathcal{E}}^{\lambda^*,\rho}_{\epsilon}(t)\leq {\mathcal{E}}^{\lambda^*,\rho}_{\epsilon}(t_2),\quad \forall\ t\geq t_2.
\end{equation*}
Since $\int^{+\infty}_{t_0}t^{\tau/2}\|\epsilon(t)\|dt<+\infty$, by similar arguments in proof of Theorem \ref{th_th1} and using the fact $C_1<0$, we obtain that ${\mathcal{E}}^{\lambda^*,\rho}(\cdot)$ and ${\mathcal{E}}^{\lambda^*,\rho}_{\epsilon}(\cdot)$ are bounded on $[t_0,+\infty)$, and then $(i),(ii),(iv)$ hold.
It follows from \eqref{eq_th2_1}, \eqref{eq_th2_4} and  the definition of ${\mathcal{E}}^{\lambda^*,\rho}(\cdot)$  that
\[\sup_{t\geq t_0} t^{(\tau-s-r)/2}\|x(t)-x^*\|<+\infty,\quad \sup_{t\geq t_0} \|\frac{1}{\delta}t^{\tau/2-s}(x(t)-x^*)+t^{\tau/2}\dot{x}(t)\|<+\infty. \]
Since $s\in[r,1]$, then
\begin{eqnarray*}
	 \sup_{t\geq t_0} t^{\tau/2}\|\dot{x}(t)\|&\leq&  \frac{1}{\delta}\sup_{t\geq t_0}t^{\tau/2-s}\|x(t)-x^*\|+\sup_{t\geq t_0} \|\frac{1}{\delta}t^{\tau/2-s}(x(t)-x^*)+t^{\tau/2}\dot{x}(t)\|\\
	& \leq&  \frac{1}{\delta}\sup_{t\geq t_0}t^{(\tau-s-r)/2}\|x(t)-x^*\|+\sup_{t\geq t_0} \|\frac{1}{\delta}t^{\tau/2-s}(x(t)-x^*)+t^{\tau/2}\dot{x}(t)\|\\
	& <& +\infty.
\end{eqnarray*}
Similarly, $\sup_{t\geq t_0} t^{\tau/2}\|\dot{\lambda}(t)\|<+\infty$,  the result $(iii)$ holds.
\end{proof}

If we take $\beta(t)$ satisfying
\[ t^s\dot{\beta}(t)\leq (\frac{1}{\delta}-(r+s) t^{s-1})\beta(t),\]
then for any $\tau\in (0,r+s)$, \eqref{ass_th2} is satisfied, and then we obtain the following results from Theorem \ref{th_th2}.

\begin{corollary}\label{cor_th2}
	Assume  that
\begin{equation}\label{ass_cor_th2}
		t^s\dot{\beta}(t)\leq (\frac{1}{\delta}-(r+s) t^{s-1})\beta(t), \qquad \int^{+\infty}_{t_0}t^{(r+s)/2}\|\epsilon(t)\|dt<+\infty.
\end{equation}
Suppose $\alpha\delta>1$ when $s=r$;  $\delta(r+s)\geq 1$ when $s=1$.
Let $(x(t),\lambda(t))$ be a global solution of the dynamic \eqref{dy_dy2}.
Then for any $(x^*,\lambda^*)\in\Omega$ and $\tau\in(0,r+s)$:
\begin{itemize}
\item [(i)]  $\|\dot{x}(t)\|+\|\dot{\lambda}(t)\| =\mathcal{O}(\frac{1}{t^{\tau/2}})$.
\item [(ii)] When $\lim_{t\to+\infty}t^{\tau}\beta(t) = +\infty,$
	\[ \mathcal{L}(x(t),\lambda^*)-\mathcal{L}(x^*,\lambda^*) =\mathcal{O}(\frac{1}{t^{\tau}\beta(t)}), \quad \|Ax(t)-b\| =\mathcal{O}(\frac{1}{t^{\tau/2}\sqrt{\beta(t)}}).\]
\end{itemize}
\end{corollary}

\begin{remark}\label{re_th2_1}
	In proof process of Theorem \ref{th_th2}, we can note that the boundedness of trajectory $(x(t),\lambda(t))$ is not guaranteed. If \eqref{ass_cor_th2} holds, we can obtain
	\[\sup_{t\geq t_0} t^{(\tau-s-r)/2}(\|x(t)-x^*\|+\|\lambda(t)-\lambda^*\|)<+\infty\]
is satisfied for any $(x^*,\lambda^*)\in\Omega$ and $\tau\in(0,r+s)$, then we get that  $t^{p}(\|x(t)-x^*\|+\|\lambda(t)-\lambda^*\|)$ is bounded for any $p<0$. When objective function $f$ satisfying  the following coercive condition:
\begin{equation}\label{eq_coercive}
	\lim_{\|x\|\to+\infty} f(x) = +\infty,
\end{equation}
we also can obtain the boundedness of $x(t)$ of dynamic \eqref{dy_dy2} from $(iv)$ of Theorem \ref{th_th2}.
\end{remark}

\begin{remark}\label{re_th2_2}
	Taking $\beta(t)\equiv 1$, $s=1$, and letting $\int^{+\infty}_{t_0}t^{(r+1)/2}\|\epsilon(t)\|dt<+\infty$ and $\delta(r+1)\geq 1$.  We obtain the $\mathcal{O}(1/t^{\tau})$ rate of convergence for any $\tau\in(0,{r+1})$, since $r\in (0,1)$, $r+1>2r$, so the results in Corollary \ref{cor_th2} improve the corresponding results in \cite[Theorem 3.4]{HeHF2020} which only obtain the $\mathcal{O}(1/t^{2r})$ convergence rate. In the case $\alpha(t)=\frac{\alpha}{t^r}$ with $\alpha>0,\ r\in(0,1)$,  the ${o}(1/t^{r+1})$ convergence of $(IGS_\alpha)$ and $(IGS_{\alpha,\epsilon})$ for problem \eqref{min_fun} have been obtained in \cite[Corollary 4.5]{AttouchC2017} and \cite[Theorem 1.2]{Balti2016} respectively, which have subtle differences  of dynamic \eqref{dy_dy2} for problem \eqref{question}.  The assumption $\int^{+\infty}_{t_0}t^{(r+1)/2}\|\epsilon(t)\|dt<+\infty$ also can find in  \cite{Balti2016}.
\end{remark}

By similar discussions in Section 2.1,  we  obtain the following optimal convergence rates of Theorem \ref{th_th2}, and the  proof is similar to Theorem \ref{th_th1_1}, so we omit it.

\begin{theorem}\label{th_th2_1}
	Let $\beta(t) =\mu \frac{e^{\frac{1}{\delta(1-s)}t^{1-s}}}{t^\tau}$ with $\mu>0, \ \tau\in(0,r+s), \ r\in(0,1),\ s\in[r,1),\ \sigma\geq 0$. Assume  $\alpha\delta>1$ when $s=r$. Suppose $\int^{+\infty}_{t_0}t^{\tau/2}\|\epsilon(t)\|dt<+\infty$. Let   $(x(t),\lambda(t))$ be a solution of dynamic \eqref{dy_dy1} and $(x^*,\lambda^*)\in\Omega$. Then
\[ |f(x(t))-f(x^*)| =\mathcal{O}\left(\frac{1}{e^{\frac{1}{\delta(1-s)}t^{1-s}}}\right),\quad \|Ax(t)-b\| =\mathcal{O}\left(\frac{1}{e^{\frac{1}{\delta(1-s)}t^{1-s}}}\right).\]
\end{theorem}

\begin{theorem}\label{th_th2_2}
Let $\beta(t)= \mu t^{\frac{1}{\delta}-\tau}$  with $\mu>0, \tau\in(0,r+1), r\in(0,1),\ s= 1,\ \delta\tau\leq 1, \ \sigma\geq 0$. Suppose $\int^{+\infty}_{t_0}t^{\tau/2}\|\epsilon(t)\|dt<+\infty$. Let   $(x(t),\lambda(t))$ be a solution of dynamic \eqref{dy_dy1} and $(x^*,\lambda^*)\in\Omega$. Then
\[ |f(x(t))-f(x^*)| =\mathcal{O}(\frac{1}{t^{1/\delta}}),\quad \|Ax(t)-b\| =\mathcal{O}(\frac{1}{t^{1/\delta}}).\]

\end{theorem}

\begin{remark}\label{re_th22_1}
	When $s=1$, taking $\tau=\frac{1}{\delta}$, then $\beta(t)=\mu>0$ is a positive constant time scaling. For any $\frac{1}{\delta}<r+1$, we can obtain the $\mathcal{O}(\frac{1}{t^{1/\delta}})$ convergence rates of objective function and constraint.
\end{remark}

\subsection{Case $r=1,\ s=1$}
Consider the case when $r=1$, $s=1$, i.e., the dynamic \eqref{dy_dynamic} becomes:
\begin{equation}\label{dy_dy3}
	\begin{cases}
		\ddot{x}(t)+\frac{\alpha}{t}\dot{x}(t) = -\beta(t)(\nabla f(x(t))+A^T(\lambda(t)+\delta t\dot{\lambda}(t))+\sigma A^T(Ax(t)-b))+\epsilon(t) ,\\
		\ddot{\lambda}(t)+\frac{\alpha}{t}\dot{\lambda}(t) = \beta(t)(A(x(t)+\delta t\dot{x}(t))-b).
	\end{cases}
\end{equation}
We will discuss dynamic \eqref{dy_dy3} with $\alpha\leq 3$ and $\alpha>3$ respectively.

\begin{theorem}\label{th_th3}
	 Assume that $\beta:[t_0,+\infty)\to(0,+\infty)$ is continuous differentiable function with
\begin{equation}\label{ass_th3}
	t\dot{\beta}(t)\leq \tau \beta(t),
\end{equation}
 and $\epsilon:[t_0,+\infty)\to\mathcal{H}_1$ satisfies
\begin{equation*}
\int^{+\infty}_{t_0}t^{\frac{\alpha-\tau}{3}}\|\epsilon(t)\|dt<+\infty.
\end{equation*}
 Let $0\leq \tau\leq \alpha \leq 3$, $\delta=\frac{3}{2\alpha+\tau}$ and $(x(t),\lambda(t))$ be a global solution of the dynamic \eqref{dy_dy3}. Then   for any $(x^*,\lambda^*)\in\Omega$, the following conclusions hold:
\begin{itemize}
	\item [(i)] $\int^{+\infty}_{t_0}t^{\frac{2(\alpha-\tau)}{3}-1}\beta(t)\|Ax(t)-b\|^2 dt <+\infty$.
	\item [(ii)] When $\tau\in[0,\alpha)$: for any $\rho\in[0,\frac{\alpha-\tau}{3})$,
	    \[\int^{+\infty}_{t_0}  t^{2\rho-1}\beta(t)(\mathcal{L}^{\sigma}(x(t),\lambda^*)-\mathcal{L}^{\sigma}(x^*,\lambda^*)) dt <+\infty, \quad \int^{+\infty}_{t_0}t^{2\rho-1}\|\dot{x}(t)\|^2+\|\dot{\lambda}(t)\|^2 dt <+\infty.\]
	\item [(iii)] When $\lim_{t\to+\infty}t^{\frac{2(\alpha-\tau)}{3}}\beta(t) = +\infty$:
	    \[ \mathcal{L}(x(t),\lambda^*)-\mathcal{L}(x^*,\lambda^*) =\mathcal{O}(\frac{1}{t^{\frac{2(\alpha-\tau)}{3}}\beta(t)}),\ \|Ax(t)-b\| =\mathcal{O}(\frac{1}{t^{\frac{\alpha-\tau}{3}}\sqrt{\beta(t)}}).\]
    \item [(iv)] When $\tau = 0$ and $\alpha = 3$:
	   \[\|\dot{x}(t)\|+\|\dot{\lambda}(t)\|=\mathcal{O}(\frac{1}{t^\rho}),\quad \forall  \ \rho\in(0,1).\]
		Otherwise:
	   \[ \|\dot{x}(t)\|+\|\dot{\lambda}(t)\|=\mathcal{O}(\frac{1}{t^{\frac{\alpha-\tau}{3}}}). \]
\end{itemize}
\end{theorem}

\begin{proof}
Given $\lambda\in\mathcal{H}_2$, 	define  $\mathcal{E}^{\lambda,\rho}(t)$ and $\mathcal{E}^{\lambda,\rho}_{\epsilon}(t)$  as \eqref{eq_A1}, \eqref{eq_A_ex} with
 $r=s=1$, $\rho\in [0,\frac{\alpha-\tau}{3}]$ and
\begin{equation}\label{eq_th3_1}
	\theta(t) = \frac{2\alpha+\tau}{3}t^{\rho-1}, \quad \eta(t) = \frac{2\alpha+\tau}{3}(1+\frac{\alpha-\tau}{3}-2\rho)t^{2\rho-2}.
\end{equation}
By computation, we have
\begin{equation}\label{eq_th3_2}
	\eta(t) \geq \frac{2\alpha+\tau}{3}(1-\frac{\alpha-\tau}{3})t^{2\rho-2}\geq 0,
\end{equation}
and  \eqref{eq_A4}, \eqref{eq_A6} are satisfied. Since $0\leq\tau\leq\alpha\leq 3$ and $\rho\in[0,\frac{\alpha-\tau}{3}]$, we also can verify that
\begin{eqnarray*}\label{eq_th3_3}
	\theta(t)\dot{\theta}(t)+\frac{\dot{\eta}(t)}{2} =\frac{2\alpha+\tau}{3}(\alpha+1-2\rho)(\rho-1)t^{2\rho-3}\leq 0.
\end{eqnarray*}
Then \eqref{eq_A3}-\eqref{eq_A6} hold for any $t\geq t_0$.

It is easy to verify that
\begin{eqnarray}\label{eq_th3_4}
	\theta(t)+\rho t^{\rho-1}-\alpha t^{\rho-r} = (\rho-\frac{\alpha-\tau}{3})t^{\rho-1}\leq 0
\end{eqnarray}
and
\begin{eqnarray}\label{eq_th3_5}
	t^{\rho}\dot{\beta}(t)+(2\rho t^{\rho-1}-\theta(t))\beta(t)&=& t^{\rho-1}(t\dot{\beta}(t)-\tau \beta(t)+ (\tau+2\rho-\frac{2\alpha+\tau}{3})\beta(t))\nonumber\\
	&\leq&  2(\rho-\frac{\alpha-\tau}{3})t^{\rho-1}\beta(t)\\
	&\leq& 0 \nonumber
\end{eqnarray}
for all $t\geq t_0$.  This together with \eqref{eq_A8} in case $\lambda=\lambda^*$ implies
\begin{eqnarray}\label{eq_th3_6}
	\dot{\mathcal{E}}^{\lambda^*,\rho}_{\epsilon}(t)	&\leq& (\rho-\frac{\alpha-\tau}{3})t^{2\rho-1}(\|\dot{x}(t)\|^2+\|\dot{\lambda}(t)\|^2) +2(\rho-\frac{\alpha-\tau}{3})t^{2\rho-1}\beta(t)(\mathcal{L}^{\sigma}(x(t),\lambda^*)- \mathcal{L}^{\sigma}(x^*,\lambda^*))\nonumber\\
	&&- \frac{\sigma t^{2\rho-1}\beta(t)}{2\delta}\|Ax(t)-b)\|^2\\
	&\leq& 0.\nonumber
\end{eqnarray}\label{eq_th3_7}
 Then ${\mathcal{E}}^{\lambda^*,\rho}_{\epsilon}(\cdot)$ is nonincreasing on $[t_0,+\infty)$,
\begin{equation*}
	{\mathcal{E}}^{\lambda^*,\rho}_{\epsilon}(t)\leq {\mathcal{E}}^{\lambda^*,\rho}_{\epsilon}(t_0),\quad \forall\ t\geq t_0.
\end{equation*}
Since $\int^{+\infty}_{t_0}t^{\frac{\alpha-\tau}{3}}\|\epsilon(t)\|dt<+\infty$,  for any $\rho\in[0,\frac{\alpha-\tau}{3}]$, we have
\[ \int^{+\infty}_{t_0}t^{\rho}\|\epsilon(t)\|dt<+\infty.\]
By similar arguments in proof of Theorem \ref{th_th1}, we obtain the boundedness of ${\mathcal{E}}^{\lambda^*,\rho}(\cdot)$ and ${\mathcal{E}}^{\lambda^*,\rho}_{\epsilon}(\cdot)$. Since  \eqref{eq_th3_6} holds for any $\rho\in[0,\frac{\alpha-\tau}{3}]$, integrating it on $[t_0,+\infty)$,  and following from the boundedness of ${\mathcal{E}}^{\lambda^*,\rho}_{\epsilon}(\cdot)$, we get the results $(i)-(ii)$.

Since ${\mathcal{E}}^{\lambda^*,\rho}(\cdot)$ is bounded for any $\rho\in[0,\frac{\alpha-\tau}{3}]$, by the definition of ${\mathcal{E}}^{\lambda^*,\rho}(\cdot)$ and \eqref{AugL}, \eqref{eq_th3_1},  we obtain $(iii)$,
\begin{equation}\label{eq_th3_8}
	\sup_{t\geq t_0} \sqrt{1+\frac{\alpha-\tau}{3}-2\rho}\times t^{\rho-1}\|x(t)-x^*\|<+\infty
\end{equation}
and
\begin{equation}\label{eq_th3_9}
	\sup_{t\geq t_0} \|\frac{2\alpha+\tau}{3}t^{\rho-1}(x(t)-x^*)+t^\rho\dot{x}(t)\|<+\infty
\end{equation}
for any $\rho\in[0,\frac{\alpha-\tau}{3}]$.

When $\tau=0$ and $\alpha = 3$: for any $\rho\in(0,1)$, $1+\frac{\alpha-\tau}{3}-2\rho = 2(1-\rho)> 0$, it follows from \eqref{eq_th3_8} and \eqref{eq_th3_9} that
\[\sup_{t\geq t_0} t^{\rho-1}\|x(t)-x^*\|<+\infty\]
and then
\begin{eqnarray*}\label{eq_th3_10}
	 \sup_{t\geq t_0}t^{\rho}\|\dot{x}(t)\|\leq 2\sup_{t\geq t_0}t^{\rho-1}\|x(t)-x^*\|+\sup_{t\geq t_0} \|2t^{\rho-1}(x(t)-x^*)+t^\rho\dot{x}(t)\|)< +\infty,
\end{eqnarray*}
for any  $\rho\in(0,1)$. Similarly $\sup_{t\geq t_0} t^{\rho}\|\dot{\lambda}(t)\|<+\infty$.

Otherwise $\alpha-\tau<3$, taking $\rho = \frac{\alpha-\tau}{3}$, then $1+\frac{\alpha-\tau}{3}-2\rho =1-\frac{\alpha-\tau}{3}>0$, by similar discussions in above, we get $(iv)$.
\end{proof}

\begin{remark}\label{re_th3_1}
Following from above proof process, when $\tau=0$ and $\alpha=3$:
\[\sup_{t\geq t_0} t^{\rho-1}(\|x(t)-x^*\|+\|\lambda(t)-\lambda^*\|)<+\infty,\quad \forall \rho\in(0,1),\]
Otherwise:
	\[\sup_{t\geq t_0} t^{\frac{\alpha-\tau}{3}-1}(\|x(t)-x^*\|+\|\lambda(t)-\lambda^*\|)<+\infty. \]
When the coercive condition \eqref{eq_coercive} satisfied , we also can obtain the boundedness of $x(t)$ of dynamic \eqref{ass_th3} with $\alpha\leq 3$.

\end{remark}

\begin{remark}\label{re_th3_2}
	 Theorem \ref{th_th3} extends the results in  \cite[Corollary 2.9]{HeHF2020} and \cite[Theorem 3.2]{Zeng2019} to general case. Taking $A=0$, $b=0$, the dynamic \eqref{dy_dy3} reduces to
	\[\ddot{x}(t)+\frac{\alpha}{t}\dot{x}(t) +\beta(t) \nabla f(x(t)) =\epsilon(t), \]
	 with $\alpha\leq 3$ for solving unconstrained optimization problem, then Theorem \ref{th_th3} also complements the results in \cite[Theorem A.1]{AttouchCRF2019}, which considered the case $\alpha\geq 3$.
\end{remark}

Taking $t\dot{\beta}(t)= \tau \beta(t)$, in which $\beta(t) = \mu t^{\tau}$ with $\mu>0$, we investigate the improved rate of convergence.

\begin{theorem}\label{th_th3_1}
Let $\beta(t)= \mu t^{\tau}$  with $\mu>0, \ 0\leq \tau\leq \alpha \leq 3$, $\delta=\frac{3}{2\alpha+\tau},\ \sigma\geq 0$. Suppose $\int^{+\infty}_{t_0}t^{(\alpha-\tau)/3}\|\epsilon(t)\|dt<+\infty$. Let   $(x(t),\lambda(t))$ be a solution of dynamic \eqref{dy_dy3}. For any $(x^*,\lambda^*)\in\Omega$:
\[ |f(x(t))-f(x^*)| =\mathcal{O}(\frac{1}{t^{(2\alpha+\tau)/3}}),\quad \|Ax(t)-b\| =\mathcal{O}(\frac{1}{t^{(2\alpha+\tau)/3}}).\]
%
\end{theorem}
\begin{proof}
	Recall the energy  functions $\mathcal{E}^{\lambda,\rho}(t)$ and $\mathcal{E}^{\lambda,\rho}_{\epsilon}(t)$  from Theorem \ref{th_th3} with $\beta(t)=\mu t^{\tau}$, $\rho=\frac{\alpha-\tau}{3}$. Then
\[ t^{\rho}\dot{\beta}(t)+(2\rho t^{\rho-1}-\theta(t))\beta(t) =0,\]
this together with \eqref{eq_th3_4} and \eqref{eq_A8} yields
\begin{eqnarray}\label{eq_th2_9}
	\dot{\mathcal{E}}^{\lambda,\rho}_{\epsilon}(t)\leq - \frac{\sigma \beta(t)}{2\delta}\|Ax(t)-b)\|^2 \leq 0, \quad \forall t\geq t_0, \lambda\in\mathcal{H}_2.
\end{eqnarray}
 By similar arguments in Theorem \ref{th_th1_1}, we obtain the results.
\end{proof}
From Theorem \ref{th_th3_1}, we obtain the following results in the case $\tau=0$ and $\tau=\alpha$, respectively.
\begin{corollary}\label{cor_th3_1}
	Let $\beta(t)= \beta>0$, $\alpha \leq 3$, $\delta=\frac{3}{2\alpha},\ \sigma\geq 0$. Suppose $\int^{+\infty}_{t_0}t^{\alpha/3}\|\epsilon(t)\|dt<+\infty$. Let   $(x(t),\lambda(t))$ be a solution of dynamic \eqref{dy_dy3}. For any $(x^*,\lambda^*)\in\Omega$:
\[ |f(x(t))-f(x^*)| =\mathcal{O}(\frac{1}{t^{2\alpha/3}}),\quad \|Ax(t)-b\| =\mathcal{O}(\frac{1}{t^{2\alpha/3}}).\]
\end{corollary}

\begin{corollary}\label{cor_th3_2}
	Let $\beta(t)= \mu t^{\alpha}$ with $\mu>0, \alpha \leq 3$, $\delta=\frac{1}{\alpha},\ \sigma\geq 0$. Suppose $\int^{+\infty}_{t_0}\|\epsilon(t)\|dt<+\infty$. Let   $(x(t),\lambda(t))$ be a solution of dynamic \eqref{dy_dy3}. For any $(x^*,\lambda^*)\in\Omega$:
\[ |f(x(t))-f(x^*)| =\mathcal{O}(\frac{1}{t^{\alpha}}),\quad \|Ax(t)-b\| =\mathcal{O}(\frac{1}{t^{\alpha}}).\]
\end{corollary}

 \begin{remark}\label{re_th31_1}
 	Taking $\beta=1$, the dynamic \eqref{dy_dy3} has been investigate in \cite{HeHF2020} and \cite{Zeng2019} for $\alpha\leq 3$.  Corollary \ref{cor_th3_1} improves the convergence rates of  \cite[Corollary 2.9]{HeHF2020} and \cite[Theorem 3.2]{Zeng2019}, which only obtain $\mathcal{O}(\frac{1}{t^{\alpha/3}})$ convergence rate of $|f(x(t))-f(x^*)|$ and $\|Ax(t)-b\|$, and it also can be viewed as analogs of the results in \cite{AttouchCRR2019,Vassilis2018}, where the convergence rate analysis of $(IGS_{\alpha,\epsilon})$ with $\alpha(t)=\frac{\alpha}{t},\ \alpha\leq 3$ for unconstrained optimization problem \eqref{min_fun}.   Corollary \ref{cor_th3_2} shows the optimal convergence rate we can expect of dynamic \eqref{dy_dy3} with $\alpha\leq 3$.
 \end{remark}

Next, we investigate  the convergence rate of dynamic \eqref{dy_dy3} with $\alpha>3$. The similar results can be found in \cite{HeHF2021C}.

\begin{theorem}\label{th_th4}
	 Assume that $\beta:[t_0,+\infty)\to(0,+\infty)$ is continuous differentiable function with
\begin{equation*}
	t\dot{\beta}(t)\leq (\frac{1}{\delta}-2) \beta(t),
\end{equation*}
and  $2\leq\frac{1}{\delta}< \alpha-1$.  Let $\epsilon:[t_0,+\infty)\to\mathcal{H}_1$ with
\begin{equation*}
\int^{+\infty}_{t_0}t\|\epsilon(t)\|dt<+\infty.
\end{equation*}
Let $(x(t),\lambda(t))$ be a global solution of the dynamic \eqref{dy_dy3} and y $(x^*,\lambda^*)\in\Omega$. Then   $(x(t),\lambda(t))$ is bounded and the following conclusions hold:
\begin{itemize}
	\item [(i)] $\int^{+\infty}_{t_0}t((\frac{1}{\delta}-2)\beta(t)-t\dot{\beta}(t))(\mathcal{L}^{\sigma}(x(t),\lambda^*)-\mathcal{L}^{\sigma}(x^*,\lambda^*)) dt <+\infty$.
	\item [(ii)] $\int^{+\infty}_{t_0}t\beta(t)\|Ax(t)-b\|^2 dt <+\infty$,\  $\int^{+\infty}_{t_0}t\|\dot{x}(t)\|^2+\|\dot{\lambda}(t)\|^2 dt <+\infty$.
	\item [(iii)] $\|\dot{x}(t)\|^2+\|\dot{\lambda}(t)\|^2=\mathcal{O}(\frac{1}{t})$.
	\item [(iv)] When $\lim_{t\to+\infty}t^2\beta(t) = +\infty$:
	\[ \mathcal{L}(x(t),\lambda^*)-\mathcal{L}(x^*,\lambda^*) =\mathcal{O}(\frac{1}{t^2\beta(t)}),\ \|Ax(t)-b\| =\mathcal{O}(\frac{1}{t\sqrt{\beta(t)}}).\]
\end{itemize}
\end{theorem}

\begin{proof}
	Given $\lambda\in\mathcal{H}_2$, 	define  $\mathcal{E}^{\lambda,\rho}(t)$ and $\mathcal{E}^{\lambda,\rho}_{\epsilon}(t)$  as \eqref{eq_A1}, \eqref{eq_A_ex} with $r=s=\rho=1$ and
\[\theta(t)  = \frac{1}{\delta},\quad \eta(t) = \frac{\alpha\delta-\delta-1}{\delta^2}.\]
Since $\alpha-1>\frac{1}{\delta}\geq 2$, by simple computations we can verify \eqref{eq_A3}-\eqref{eq_A6}. It follows from assumptions that
\begin{eqnarray*}
	t^{\rho}\dot{\beta}(t)+(2\rho t^{\rho-1}-\theta(t))\beta(t)= t\dot{\beta}(t)+(2-\frac{1}{\delta})\beta(t)\leq 0.
\end{eqnarray*}
Taking $\lambda=\lambda^*$, this together with \eqref{eq_A8}  implies
\begin{eqnarray*}
	\dot{\mathcal{E}}^{\lambda^*,\rho}_{\epsilon}(t) 	&\leq& (\frac{1}{\delta}+1-\alpha)t (\|\dot{x}(t)\|^2+\|\dot{\lambda}(t)\|^2)+t(t\dot{\beta}(t)+(2-\frac{1}{\delta})\beta(t))(\mathcal{L}^{\sigma}(x(t),\lambda^*)- \mathcal{L}^{\sigma}(x^*,\lambda^*))\\
	&&- \frac{\sigma t\beta(t)}{2\delta}\|Ax(t)-b\|^2\nonumber\\
	&\leq & 0. \nonumber
\end{eqnarray*}
By similarly arguments in proof of  Theorem \ref{th_th1}, we obtain the boundedness of  ${\mathcal{E}}^{\lambda^*,\rho}(\cdot)$ and ${\mathcal{E}}^{\lambda^*,\rho}_{\epsilon}(\cdot)$. This yields  $(i),(ii),(iv)$.
Since $\eta(t)=\frac{\alpha\delta-\delta-1}{\delta^2}>0$, we get that $(x(t),\lambda(t))$ is bounded and
\[\sup_{t\geq t_0} t(\|\dot{x}(t)\|^2+\|\lambda(t)\|^2)<+\infty.\]
This implies $(iii)$.
\end{proof}

\begin{remark}\label{re_th4_1}
	Theorem \ref{th_th4} extends the results in \cite[Theorem A.1]{AttouchCRF2019} and \cite[Section 3.2]{Attouchcrf2019} from $(IGS_{\alpha,\epsilon})$ with $\alpha(t)=\frac{\alpha}{t},\ \alpha>3$ for  problem \eqref{min_fun} to primal-dual dynamic for problem \eqref{question}. Taking $\beta(t)\equiv 1$, we recover the convergence rate of \cite[Corollary 2.9]{HeHF2020} and \cite[Theorem 3.1]{Zeng2019}, moreover when $A=0,\ b=0$, we get the classical results for $(IGS_{\alpha})$ and $(IGS_{\alpha,\epsilon})$ with $\alpha(t)=\frac{\alpha}{t}$ with $\alpha>3$, which can be seen as a continuous version of the Nesterov method, see \cite{AttouchCPR2018,Aujol2019,May2015,Su2014}.
\end{remark}

Let $t\dot{\beta}(t)= (\frac{1}{\delta}-2) \beta(t)$. We have  $\beta(t) = \mu t^{\frac{1}{\delta}-2}$ with $\mu>0$.  By similar proof of  Theorem \ref{th_th1_1}, we  obtained following results, and the corresponding results of unperturbed case can be found in \cite[Proposition 6.3]{AttouchBCR2021}.

\begin{theorem}\label{th_th4_1}
Let $\beta(t)= \mu t^{1/\delta-2}$  with $\mu>0, \ 2\leq \frac{1}{\delta}< \alpha-1$, $\sigma\geq 0$. Suppose $\int^{+\infty}_{t_0}t\|\epsilon(t)\|dt<+\infty$. Let   $(x(t),\lambda(t))$ be a solution of dynamic \eqref{dy_dy3}. For any $(x^*,\lambda^*)\in\Omega$:
\[ |f(x(t))-f(x^*)| =\mathcal{O}(\frac{1}{t^{1/\delta}}),\quad \|Ax(t)-b\| =\mathcal{O}(\frac{1}{t^{1/\delta}}).\]
\end{theorem}
From Theorem \eqref{th_th4_1}, we have following result.

\begin{corollary}\label{cor_th4_1}
	Let $\beta(t)= \beta>0$,  $\delta=\frac{1}{2}$, $\alpha>3$, $\sigma\geq 0$. Suppose $\int^{+\infty}_{t_0}t\|\epsilon(t)\|dt<+\infty$. Let   $(x(t),\lambda(t))$ be a solution of dynamic \eqref{dy_dy3}. For any $(x^*,\lambda^*)\in\Omega$:
\[ |f(x(t))-f(x^*)| =\mathcal{O}(\frac{1}{t^2}),\quad \|Ax(t)-b\| =\mathcal{O}(\frac{1}{t^2}).\]
\end{corollary}

\begin{remark}\label{re_th41_1}
	Theorem \ref{th_th4_1} shows the optimal convergence rates of dynamic \eqref{dy_dy3} in the case $\alpha>3$. The $O(1/t^{p+2})$ convergence rate results associated with the time scaling $\beta(t) = \mu t^p$ for unconstrained optimization problem \eqref{min_fun} can be found in \cite{AttouchCRF2019,Wibisono2016}, it also can be found in \cite{Fazlyab2017} with Euclidean setting of Bregman distance for problem \eqref{question}. Corollary \ref{cor_th4_1} showst the convergence rate of  objective function and  constraint of dynamical system \eqref{dy_zeng} is $\mathcal{O}(\frac{1}{t^2})$ instead of $\mathcal{O}(\frac{1}{t})$.
\end{remark}

\subsection{Summary of results}
In the subsection, we complete the tables giving a synthetic view of convergence results in before.

For dynamic \eqref{dy_Appendix} with  different $r$ and $s$, chose suitable parameters $\alpha,\ \delta$.
  Table \ref{table_1} lists the convergences rates for $\mathcal{L}(x(t),\lambda^*)-\mathcal{L}(x^*,\lambda^*) $ of dynamic \eqref{dy_dynamic} under different  assumptions of $\beta(t)$ and $\epsilon(t)$.   Table \ref{table_2}  summarizes the properties of trajectory $(x(t),\lambda(t))$ and its derivates  $(\dot{x}(t),\dot{\lambda}(t))$.  (See Theorem \ref{th_th1}, Corollary \ref{cor_th2}, Theorem \ref{th_th3}, Theorem \ref{th_th4}, Remark \ref{re_th2_1}, Remark \ref{re_th3_1}). The results extend the  inertial dynamic with time scaling in \cite{ AttouchCRF2019,Attouchcrf2019,Balhag2020,Wibisono2016} for problem \eqref{min_fun} to primal-dual dynamic \eqref{dy_dynamic} for problem \eqref{question}. Taking $A=0, b=0$, our results also can  complement the existing results the inertial dynamic with time scaling.

  \begin{table}[!h]
		\centering
		\caption{Convergence rates for $\mathcal{L}(x(t),\lambda^*)-\mathcal{L}(x^*,\lambda^*)$ of dynamic \eqref{dy_dynamic}}\label{table_1}
		\renewcommand\arraystretch{1.5}
		\setlength{\tabcolsep}{2.5mm}
		\begin{tabular}{|c|c|c|c|c|}
			\hline
			 \multicolumn{2}{|c|}{$r, s$}  & $\beta(t)$ &$\epsilon(t)$&$\mathcal{L}(x(t),\lambda^*)-\mathcal{L}(x^*,\lambda^*) $\\
			 \hline
			\multicolumn{2}{|c|}{$r =0, s\in[0,1]$}&{$t^s\dot{\beta}(t)\leq (\frac{1}{\delta}-st^{s-1})\beta(t)$}&$\int^{+\infty}_{t_0}t^{\frac{s}{2}}\|\epsilon(t)\|dt<+\infty$ &{$\mathcal{O}(\frac{1}{t^s\beta(t)})$} \\
			\hline
			\multicolumn{2}{|c|}{$r\in(0,1), s\in[r,1]$} & {$t^s\dot{\beta}(t)\leq (\frac{1}{\delta}-(r+s) t^{s-1})\beta(t)$}&{$\int^{+\infty}_{t_0}t^{\frac{r+s}{2}}\|\epsilon(t)\|dt<+\infty$} &{$\mathcal{O}(\frac{1}{t^{\rho}\beta(t)}), \forall \rho\in(0,{r+s})$}\\
			\hline
			 \multirow{2}{*}{$r=s=1$}&{$\alpha\leq 3$}&{$	t\dot{\beta}(t)\leq \tau \beta(t)$, $\tau\in[0,\alpha]$}& {$\int^{+\infty}_{t_0}t^{\frac{\alpha-\tau}{3}}\|\epsilon(t)\|dt<+\infty$}& {$\mathcal{O}\left(\frac{1}{t^{{2(\alpha-\tau)/3}}\beta(t)}\right)$}\\
			 \cline{2-5}
			&{$\alpha> 3$} & {$ t\dot{\beta}(t)\leq (\frac{1}{\delta}-2) \beta(t)$}& {$\int^{+\infty}_{t_0}t\|\epsilon(t)\|dt<+\infty$}& {$\mathcal{O}(\frac{1}{t^{2}\beta(t)})$}\\	
			\hline
		\end{tabular}	
\end{table}

\begin{table}[!h]
		\centering
		\caption{Summary of trajectory properties}\label{table_2}
		\renewcommand\arraystretch{1.4}
		\setlength{\tabcolsep}{5.5mm}
		\begin{tabular}{|c|c|c|c|}
			\hline
			 \multicolumn{2}{|c|} {$r, s$}  &$\|\dot{x}(t)\|+\|\dot{\lambda}(t)\|$ &$\mathcal{I}=\|{x}(t)-x^*\|+\|\lambda(t)-\lambda^*\|$ \\
			 \hline
			\multicolumn{2}{|c|} {$r = s = 0$}& bounded &  $\mathcal{I}$ bounded\\
			\hline
			\multicolumn{2}{|c|}{$r = 0,\ s\in(0,1]$}& {$\mathcal{O}(\frac{1}{t^{s/2}})$}& $\mathcal{I}$ bounded \\
		\hline
		  \multicolumn{2}{|c|} {$r\in(0,1),\ s\in[r,1]$}& $\mathcal{O}(\frac{1}{t^{\rho}}),\forall \rho\in(0,\frac{r+s}{2})$  &$t^\rho \mathcal{I}$ bounded, $ \forall \rho\in(-\frac{r+s}{2}, 0)$\\
		  \hline
		\multirow{3}{*}{$r=s=1$} & $\alpha=3, \tau=0$ &$\mathcal{O}(\frac{1}{t^{\rho}}),\forall \rho\in(0,1)$ &$t^\rho \mathcal{I}$ bounded, $ \forall \rho\in(-1,0)$ \\
		 \cline{2-4}
		   &$\alpha\leq 3,\ 0\leq\alpha-\tau<3$ & $\mathcal{O}(\frac{1}{t^{(\alpha-\tau)/3}})$ &$t^{\frac{\alpha+\tau}{3}-1}\mathcal{I}$ bounded \\
		 \cline{2-4}
		 &  $\alpha>3$ &$\mathcal{O}(\frac{1}{t})$& $\mathcal{I}$ bounded\\
\hline
		\end{tabular}	
\end{table}

Select a specific time scaling $\beta(t)$ with suitable parameters $\alpha,\ \delta$. Table \ref{table_3} shows optimal convergence rates we can expect for different choices of  coefficients. (See Theorem \ref{th_th1_1}, Theorem \ref{th_th1_2}, Theorem \ref{th_th1_3}, Theorem \ref{th_th2_1}, Theorem \ref{th_th2_2}, Corollary \ref{cor_th3_2}, Theorem \ref{th_th4_1})

\begin{table}[!h]
		\centering
		\caption{Optimal convergence rates of $|f(x(t)-f(x^*))|$ and $\|Ax(t)-b\|$}\label{table_3}
		\renewcommand\arraystretch{1.6}
		\setlength{\tabcolsep}{7mm}
		\begin{tabular}{|c|c|c|c|}
			\hline
			  \multicolumn{2}{|c|}{$r, s$}  &$\beta(t)$ &$|f(x(t)-f(x^*))|$ and $\|Ax(t)-b\|$\\
			 \hline
			 \multicolumn{2}{|c|}{$r =0, s\in [0,1)$}&{$\mu \frac{e^{\frac{1}{\delta(1-s)}t^{1-s}}}{t^s}$} & {$\mathcal{O}\left(\frac{1}{e^{\frac{1}{\delta(1-s)}t^{1-s}}}\right)$ } \\
			\hline
			\multicolumn{2}{|c|}{$r =0, s=1$}&{$\mu t^{\frac{1}{\delta}-1}$} & {$\mathcal{O}(\frac{1}{t^{1/\delta}})$} \\
			\hline
			\multicolumn{2}{|c|}{$r\in(0,1), s\in [r,1)$}&{$\mu \frac{e^{\frac{1}{\delta(1-s)}t^{1-s}}}{t^\tau}$, $\forall \tau\in(0,r+s)$} & {$\mathcal{O}\left(\frac{1}{e^{\frac{1}{\delta(1-s)}t^{1-s}}}\right)$} \\
			\hline
			\multicolumn{2}{|c|}{$r\in(0,1), s=1$}&{$\mu t^{\frac{1}{\delta}-\tau}$, $\forall \tau\in(0,r+1)$} & {$\mathcal{O}(\frac{1}{t^{1/\delta}})$} \\
			\hline
			\multirow{2}{*}{$r=s=1$}& $\alpha\leq 3$& $\mu t^{\alpha}$& {$\mathcal{O}(\frac{1}{t^\alpha})$}\\
			\cline{2-4}
			&$\alpha> 3$& $\mu t^{1/\delta-2}$& {$\mathcal{O}(\frac{1}{t^{1/\delta}})$}\\
			\hline
		\end{tabular}	
\end{table}

Taking time scaling $\beta(t)\equiv 1$, Table \ref{table_4} lists the  corresponding convergence rates (See Remark \ref{re_th1_3}, Theorem \ref{th_th1}, Theorem \ref{th_th1_3}, Theorem \ref{th_th2}, Theorem \ref{th_th2_2}, Corollary \ref{cor_th3_1}, Corollary \ref{cor_th4_1}), it extends the convergence rates of $(IGS_\alpha)$ and  $(IGS_{\alpha,\epsilon})$ in \cite{AttouchCPR2018, AttouchCRR2019, Balti2016, Su2014, Sun2019, Vassilis2018} for unconstrained optimization problems to primal-dual dynamic \eqref{dy_dynamic} for linear equality  constrained optimization problems. It also  extend and complements the existing results of inertial primal-dual dynamic in \cite{AttouchCFR2021,HeHF2021C,HeHF2020,HeHF2021F,Zeng2019}.

 \begin{table}[!h]
		\centering
		\caption{Convergence rates  of dynamic \eqref{dy_dynamic} with $\beta(t)\equiv 1$}\label{table_4}
		\renewcommand\arraystretch{1.4}
		\setlength{\tabcolsep}{8mm}
		\begin{tabular}{|c|c|c|c|}
			\hline
			 \multicolumn{2}{|c|}{$r, s$}  & $|f(x(t))-f(x^*)|$ and $\|Ax(t)-b\|$&$\mathcal{L}(x(t),\lambda^*)-\mathcal{L}(x^*,\lambda^*) $\\
			 \hline
			 \multicolumn{2}{|c|}{$r =0, s=0$}&{$\mathcal{O}(\frac{1}{\sqrt{t}})$ ergodic sence}&{$\mathcal{O}(\frac{1}{{t}})$ ergodic sence} \\
			\hline
			\multicolumn{2}{|c|}{$r =0, s\in(0,1)$}&{$\mathcal{O}(\frac{1}{t^{s/2}})$} &{$\mathcal{O}(\frac{1}{t^s})$} \\
			\hline
			\multicolumn{2}{|c|}{$r=0, s=1$} &\multicolumn{2}{|c|}{$\mathcal{O}(\frac{1}{t})$}\\
			\hline
			\multicolumn{2}{|c|}{$r\in(0,1), s\in[r,1)$} &{$\mathcal{O}(\frac{1}{t^{\tau/2}}), \forall \tau\in(0,r+s)$} &{$\mathcal{O}(\frac{1}{t^{\tau}}), \forall \tau\in(0,r+s)$}\\
			\hline
			\multicolumn{2}{|c|}{$r\in(0,1), s=1$}  &\multicolumn{2}{|c|}{$\mathcal{O}(\frac{1}{t^{\tau}}),\quad \forall \tau\in(0,r+1)$}\\
			\hline
			 \multirow{2}{*}{$r=s=1$}&{$\alpha\leq 3$}&\multicolumn{2}{|c|}{$\mathcal{O}(\frac{1}{t^{{2\alpha/3}}})$}\\
			 \cline{2-4}
			&{$\alpha> 3$} &\multicolumn{2}{|c|}{$\mathcal{O}(\frac{1}{t^2})$}\\	
			\hline
		\end{tabular}	
\end{table}

\section{Conclusion}
In this paper,  we propose a family of damped inertial primal-dual dynamical systems with time scaling for solving problem \eqref{question} in Hilbert space. We extend the  inertial dynamic  in \cite{AttouchCPR2018, AttouchCRR2019, AttouchCRF2019,Balhag2020, Balti2016, Su2014, Sun2019, Vassilis2018,Wibisono2016} for solving unconstrained optimization problems to primal-dual dynamic \eqref{dy_dynamic} for  solving linear equality constrained convex optimization problems. Our results also  extend and complement the existing results of inertial primal-dual dynamics in \cite{AttouchCFR2021,HeHF2021C,HeHF2020,HeHF2021F,Zeng2019}.
 Taking $A=0, b=0$,  our results also complement the convergence rate results of  existing inertial dynamic for solving unconstrained convex optimization problems. By discretization of primal-dual dynamic \eqref{dy_Appendix}, it may lead to new primal-dual algorithms for solving problem \eqref{question},  how to chose suitable discretization scheme of \eqref{dy_Appendix} to get rate-matching algorithms is an interesting direction of research. From references \cite{HeHF2021C,HeHF2021F}, it seems achievable, and we will consider it in the future works.

\appendix
\section{Some auxiliary results}
\subsection{Differentiating the energy function}\label{APPENDIX_A1}
In this part, we list the main calculation procedures for differentiating the energy function $\mathcal{E}^{\lambda,\rho}_{\epsilon}(t)$.

Multiplying the first equation of \eqref{dy_Appendix} by $t^\rho$, we have
\[t^\rho\ddot{x}(t) = -\alpha t^{\rho-r}\dot{x}(t)- t^\rho\beta(t)(\nabla f(x(t))+A^T(\lambda(t)+\delta t^s\dot{\lambda}(t))+\sigma A^T(Ax(t)-b))+t^\rho\epsilon(t). \]
This yields
\begin{eqnarray*}
	&&\dot{\mathcal{E}}_1(t) = \langle \theta(t)(x(t)-x^*)+t^\rho\dot{x}(t), \dot{\theta}(t)(x(t)-x^*)+\theta(t)\dot{x}(t)+\rho t^{\rho-1}\dot{x}(t)+t^\rho\ddot{x}(t)\rangle \\
	&&\qquad + \frac{\dot{\eta}(t)}{2}\|x(t)-x^*\|^2+ \eta(t) \langle x(t)-x^*,\dot{x}(t)\rangle\\
	&&\quad = \langle \theta(t)(x(t)-x^*)+t^\rho\dot{x}(t), \dot{\theta}(t)(x(t)-x^*)+(\theta(t)+\rho t^{\rho-1}-\alpha t^{\rho-r})\dot{x}(t)\\
	&&\qquad -t^\rho\beta(t)(\nabla f(x(t))+A^T(\lambda(t)+\delta t^s\dot{\lambda}(t))+\sigma A^T(Ax(t)-b))+t^\rho \epsilon(t) \rangle \\
	&&\qquad + \frac{\dot{\eta}(t)}{2}\|x(t)-x^*\|^2+ \eta(t) \langle x(t)-x^*,\dot{x}(t)\rangle\\
	&&\quad= (\theta(t)\dot{\theta}(t)+\frac{\dot{\eta}(t)}{2})\|x(t)-x^*\|^2+ t^\rho(\theta(t)+\rho t^{\rho-1}-\alpha t^{\rho-r})\|\dot{x}(t)\|^2\\
	&&\qquad+(\theta(t)(\theta(t)+\rho t^{\rho-1}-\alpha t^{\rho-r})+t^\rho\dot{\theta}(t)+\eta(t))\langle x(t)-x^*,\dot{x}(t)\rangle \\
	&&\qquad-\delta\theta(t)t^{\rho+s}\beta(t)\langle x(t)-x^*, A^T\dot{\lambda}(t)\rangle
	-\delta t^{2\rho+s}\beta(t)\langle A\dot{x}(t),\dot{\lambda}(t)\rangle\\
	&& \qquad-\theta(t)t^\rho\beta(t)(\langle x(t)-x^*, \nabla f(x(t))+A^T\lambda(t)+\sigma A^T(Ax(t)-b)\rangle\\
	&&\qquad- t^{2\rho}\beta(t)\langle \dot{x}(t), \nabla f(x(t))+A^T\lambda(t)+\sigma  A^T(Ax(t)-b)\rangle\\
	&&\qquad +\langle \theta(t)(x(t)-x^*)+t^\rho\dot{x}(t),t^\rho \epsilon(t)\rangle.
	\end{eqnarray*}
Similarly, we have
\begin{eqnarray*}
	&&\dot{\mathcal{E}}_2(t)= (\theta(t)\dot{\theta}(t)+\frac{\dot{\eta}(t)}{2})\|\lambda(t)-\lambda\|^2+ t^\rho(\theta(t)+\rho t^{\rho-1}-\alpha t^{\rho-r})\|\dot{\lambda}(t)\|^2\\
	&&\quad+(\theta(t)(\theta(t)+\rho t^{\rho-1}-\alpha t^{\rho-r})+t^\rho\dot{\theta}(t)+\eta(t))\langle \lambda(t)-\lambda,\dot{\lambda}(t)\rangle \\
	&&\quad +\theta(t)t^\rho\beta(t)\langle  \lambda(t)-\lambda, Ax(t)-b\rangle+\delta\theta(t)t^{\rho+s}\beta(t)\langle \lambda(t)-\lambda, A\dot{x}(t)\rangle\\
	&&\quad +t^{2\rho}\beta(t)\langle \dot{\lambda}(t),  Ax(t)-b\rangle+ \delta t^{2\rho+s}\beta(t)\langle \dot{\lambda}(t),A\dot{x}(t)\rangle.
\end{eqnarray*}
Differentiating of $\mathcal{E}_0(t)$ to get
\begin{eqnarray*}
	\dot{\mathcal{E}}_0(t) &=& t^{2\rho}\beta(t)\langle \nabla f(x(t))+A^T\lambda+\sigma A^T(Ax(t)-b),\dot{x}(t)\rangle\\
	&&+(2\rho t^{2\rho-	1}\beta(t)+t^{2\rho}\dot{\beta}(t))( \mathcal{L}^{\sigma}(x(t),\lambda)- \mathcal{L}^{\sigma}(x^*,\lambda)).
\end{eqnarray*}
Let $\theta(t)$ satisfy $t^{2\rho}\beta(t)=\delta \theta(t)t^{\rho+s}\beta(t)$. Adding $\dot{\mathcal{E}}_0(t)$, $\dot{\mathcal{E}}_1(t)$, $\dot{\mathcal{E}}_2(t)$ together, using $Ax^*= b$  and rearranging the terms, we get
\begin{equation*}
	\dot{\mathcal{E}}^{\lambda,\rho}(t)=\dot{\mathcal{E}}_0(t)+	\dot{\mathcal{E}}_1(t)+	\dot{\mathcal{E}}_2(t) = \sum^5_{i=1} \mathcal{V}_i(t),
\end{equation*}
where
\begin{eqnarray*}
	\mathcal{V}_1(t) &=&  \left(\theta(t)\dot{\theta}(t)+\frac{\dot{\eta}(t)}{2}\right)(\|x(t)-x^*\|^2+\|\lambda(t)-\lambda\|^2),\\
	\mathcal{V}_2(t) &=&(\theta(t)(\theta(t)+\rho t^{\rho-1}-\alpha t^{\rho-r})+t^\rho\dot{\theta}(t)+\eta(t))(\langle x(t)-x^*,\dot{x}(t)\rangle+\langle \lambda(t)-\lambda,\dot{\lambda}(t)\rangle),\\
	\mathcal{V}_3(t) &=& t^\rho(\theta(t)+\rho t^{\rho-1}-\alpha t^{\rho-r}) (\|\dot{x}(t)\|^2+\|\dot{\lambda}(t)\|^2),\\
	\mathcal{V}_4(t) &=& t^{\rho}(t^{\rho}\dot{\beta}(t)+(2\rho t^{\rho-1}-\theta(t))\beta(t))(\mathcal{L}^{\sigma}(x(t),\lambda)- \mathcal{L}^{\sigma}(x^*,\lambda))\\
	&&+\theta(t)t^\rho\beta(t)( f(x(t))-f(x^*)-\langle x(t)-x^*, \nabla f(x(t))\rangle)- \frac{\sigma\theta(t)t^\rho\beta(t)}{2}\|Ax(t)-b\|^2,\\
	\mathcal{V}_5(t) &=& \langle \theta(t)(x(t)-x^*)+t^\rho\dot{x}(t),t^\rho \epsilon(t)\rangle.
\end{eqnarray*}
To investigate the rates of convergence of dynamical system \eqref{dy_Appendix}, we need to find the appropriate $\theta(t)$ and $\eta(t)$ to satisfy the following conditions:
\begin{eqnarray}
		\theta(t)\geq 0, \quad \eta(t)&\geq& 0,\label{eq_A3}\\
		t^{2\rho}\beta(t)-\delta \theta(t)t^{\rho+s}\beta(t)&=& 0,\label{eq_A4}\\
		\theta(t)\dot{\theta}(t)+\frac{\dot{\eta}(t)}{2}&\leq& 0,\label{eq_A5}\\
		\theta(t)(\theta(t)+\rho t^{\rho-1}-\alpha t^{\rho-r})+t^\rho\dot{\theta}(t)+\eta(t)&=&0,\label{eq_A6}
	\end{eqnarray}
Then $\mathcal{V}_1\leq 0$,  $\mathcal{V}_2=0$, this together with the convexity of $f$ yields
\begin{eqnarray}\label{eq_A8}
\dot{\mathcal{E}}^{\lambda,\rho}_{\epsilon}(t) &=&\dot{\mathcal{E}}^{\lambda,\rho}(t)-\langle \theta(t)(x(t)-x^*)+t^\rho\dot{x}(t),t^\rho \epsilon(t)\rangle\nonumber \\
		&\leq& t^\rho(\theta(t)+\rho t^{\rho-1}-\alpha t^{\rho-r}) (\|\dot{x}(t)\|^2+\|\dot{\lambda}(t)\|^2)- \frac{\sigma\theta(t)t^\rho\beta(t)}{2}\|Ax(t)-b)\|^2 \nonumber\\
	 && +t^{\rho}(t^{\rho}\dot{\beta}(t)+(2\rho t^{\rho-1}-\theta(t))\beta(t))(\mathcal{L}^{\sigma}(x(t),\lambda)- \mathcal{L}^{\sigma}(x^*,\lambda))
\end{eqnarray}
for any $\lambda\in\mathcal{H}_2$.

\subsection{Technical lemmas:}
In convergence analysis for the dynamical system, we shall recall the following  lemmas.
\begin{lemma}\cite[Lemma A.5]{Brezis1973}\label{le_A1}
	Let $\nu:[t_0,T]\to [0,+\infty)$ be integrable, and $M\geq 0$. Suppose $\mu:[t_0,T]\to R $ is continuous and
	\[ \frac{1}{2}\mu(t)^2\leq \frac{1}{2}M^2+\int^{t}_{t_0}\nu(s)\mu(s)ds\]
	for all $t\in[t_0,T]$. Then $|\mu(t)|\leq M+\int^t_{t_0}\nu(s)ds$ for all $t\in[t_0,T]$.
\end{lemma}

\begin{lemma}\cite[Lemma 2.1.]{Xu2017}\label{le_A2}
	For problem \eqref{question}, let $x^*$ be a solution. Given a function $\phi$ and a fix point $x$, if for any $\lambda$ it holds that
	\[ f(x)-f(x^*)+\langle \lambda,Ax-b\rangle \leq \phi(\lambda),\]
	then for any $\varrho>0$, we have
	\[ f(x)-f(x^*)+\varrho\|Ax-b\| \leq \sup_{\|\lambda\|\leq \varrho}\phi(\lambda),\]
\end{lemma}

\end{document}